\documentclass[12pt,twoside,leqno,final]{amsart}
\usepackage{amsmath}
\usepackage{amssymb}

\theoremstyle{plain}
\newtheorem{thm}{Theorem}[section]
\newtheorem{lem}[thm]{Lemma}
\newtheorem{defn}[thm]{Definition}
\newtheorem{prop}[thm]{Proposition}
\newtheorem{cor}[thm]{Corollary}
\newtheorem{rem}[thm]{Remark}

\newcommand{\C}{{\mathbb C}}
\newcommand{\R}{{\mathbb R}}

\newcommand{\B}{{\mathbb B}}
\newcommand{\bB}{{\mathbb S}}

\newcommand{\T}{{\mathcal T}}
\newcommand{\PP}{\mathbb{P}}
\newcommand{\PN}{{\mathcal P}}
\newcommand{\e}{\varepsilon}
\newcommand{\p}{\partial}
\newcommand{\z}{\zeta}

\title[On multipliers for Hardy-Sobolev spaces]{On multipliers for Hardy-Sobolev spaces  \\and holomorphic potentials}

\author{Carme Cascante, Joan F\`abrega and Joaqu\'\i n M. Ortega}
\address{Dept.\ Matem\`atica Aplicada i An\`alisi, Universitat  de Barcelona, Gran Via 585, 08071 Barcelona, Spain}
\email{cascante@ub.edu,\quad joan$_{-}$fabrega@ub.edu,\quad ortega@ub.edu}

\date{\today}
\thanks{Partially supported by  DGICYT Grant MTM2011-27932-C02-01  and DURSI Grant 2009SGR 1303.}

\begin{document}

\begin{abstract} 
We study the action of some generalized integral operators of Bergman type on pointwise multipliers of holomorphic Triebel-Lizorkin spaces. 
We  construct  nontrivial examples of pointwise multipliers in Hardy-Sobolev spaces and give applications of all these results.
\end{abstract}

\maketitle

\section{Introduction}\label{sec:intro} 
The main object of this paper is the study of the space of the pointwise multipliers of Hardy-Sobolev spaces in the unit ball $\B$ of $\C^n$. In particular, we study the action of some integral operators that generalize the Bergman type operators, on
these spaces of mutipliers and we construct interesting nontrivial examples of such multipliers.

We recall that 
if $1\le p<\infty$ and $s\in \R$, then 
the Hardy-Sobolev space $H^p_s$ consists of the holomorphic functions on $\B$ such that 
 if ${\displaystyle f(z)=\sum_k f_k(z)}$ is its homogeneous polynomial expansion, and   the fractional radial derivative is defined by $$ {\displaystyle (1+R)^s f(z):=\sum_k (1+k)^s f_k(z)}, $$ 
 then
$$
\|f\|_{H^p_s}:=\|(1+R)^{s}f\|_{H^p}<\infty.
$$ 
It is well known that any pointwise multiplier of $H_s^p$ is a function in $H^\infty$. If $s\leq 0$, the space of multipliers coincide with $H^\infty$. If $s>n/p$, then the space $H_s^p$ is a multiplicative algebra and, in consequence, it coincides with the algebra of its multipliers. In what follows, we will assume that $0<s\leq n/p$.

If $1<p$ and $0<s<n$, a function $f$ belongs to $H_s^p$ if and only if it can be represented as $\displaystyle{f(z)={\mathcal C}_s(h)(z):= \int_{\bB} \frac{h(\zeta)}{(1-z\overline{\zeta})^{n-s}}}d\sigma(\zeta)
$, 
for some $h\in L^p(\bB)$, where $d\sigma$ is the normalized Lebesgue measure on $\bB$. 
The functions $h$ can be restricted to be boundary values of $H^p$ functions. In this case, ${\mathcal C}_s$ is a bijective operator from $H^p$ to $H_s^p$.

 The first motivation of this paper is, under the assumption of the boundedness of  $g$, to obtain a description of the functions $h\in H^p$ such that $g={\mathcal C}_s(h)$ is  a pointwise multiplier for $H_s^p$.
For such reason we construct an inverse of the operator ${\mathcal C}_s$, which will be given in terms of sums of a generalization of integral operators of Bergman type of the form 
$$
\PN^{N,M}(f)(z)=c_N\int_{\B} \frac{(1-|w|^2)^{N-1}}{(1-z\overline{w})^{n+M}}f(w)d\nu(w),
\quad N>0,\,M>-n.
$$ 
Here   $d\nu$  is the normalized Lebesgue measure on the unit ball $\B$ and $c_N$ is a constant such that $\PN^{N,N}$ is a reproducing kernel for holomorphic functions on $\overline{\B}$. These operators are extended to the case $N=0$ and $M>-n$ by $\PN^{0,M}={\mathcal C}_{-M}$.

These operators, for the particular case where $M=N+k$ with $k$ a positive integer, coincide with differential operators of order $k$ on the holomorphic functions. Since the technical difficulties to handle the space of multipliers of $H_s^p$ are  similar to the ones existing in the context of multipliers of holomorphic Triebel-Lizorkin spaces $F_s^{p,q}$, which will be introduced in Section \ref{sec:preliminaries}, we have chosen to work in this most general framework that allows, for instance, to obtain results simultaneously on Hardy-Sobolev spaces $H_s^p$ ($F_s^{p,q}$ with $q=2$) and on holomorphic Besov spaces $B_s^p$ ($F_s^{p,q}$ with $q=p$).

Given  $X$ and $Y$ Banach spaces of holomorphic functions, we denote by $Mult(X\to Y)$ the space of pointwise multipliers from $X$ to $Y$. If $X=Y$, then we simply write $Mult(X)$.

Our first result deals with the action of the operator $\PN^{N,M}$ on $Mult(F^{p,q}_s\to F^{p,q}_{s'})$, when $s'<s$. In the case $s=s'$, the space  $Mult(F^{p,q}_s)$  must be substituted by the space   
 $$
CF^{p,q}_s:=\left\{g\in F^{p,q}_s:\,\|g\|_{CF^{p,q}_s}<\infty\right\},
$$
where
$
\|g\|_{CF^{p,q}_s}:=\left\|(1+R)^{k_s} g\right\|_{Mult(F^{p,q}_s\to F^{p,q}_{s-k_s})},
$
$k_s=\max([s],0)+1$ and $[s]$ is the integer part of $s$.

Clearly,
$g\in CF^{p,q}_s$ if and only if $d\mu_g(z):= |(1+R)^{k_s} g(z)|^q(1-|z|^2)^{(k_s-s)q} d\nu(z)$ is a Carleson measure for $F^{p,q}_s$
 in the sense $F^{p,q}_s\subset T^{p,q}(\mu_g)$.
 
 Here $T^{p,q}(\mu_g)$ is a tent space that will be defined in Section \ref{sec:preliminaries}. 
 
A standard technique to study differential operators on a space of multipliers, uses  Leibnitz's formula $f Rg= R(fg)-gRf$. Here, in this more general context on the study of the action of the operators $\PN^{N,M}$ on the spaces of multipliers, we prove a generalized Leibnitz's type formula for $\PN^{N,M}$, which has a complementary error term. This formula allows us to prove the following result:

\begin{thm}\label{thm:bijectMF}
Let $1< p,q<\infty$, $0\le s\le n/p$, $N\ge 0$, $M>-n$ and $s'<  s$, satisfying either 
$-N<s'$ or $N= s'=0$ and $q\le 2$.

Then we have:
\begin{enumerate}
	\item \label{item:bijectMF1} If $N-M< s-s'$, then $\PN^{N,M}$ is a bijective map from 
$$Mult(F^{p,q}_s\to F^{p,q}_{s'}) \quad\text{to}\quad  
Mult(F^{p,q}_s\to F^{p,q}_{s'+N-M}).$$ 

 \item \label{item:bijectMF3} If $N-M<0$, then $\PN^{N,M}$ is a bijective map from 
$$CF^{p,q}_s\quad\text{to}\quad Mult(F^{p,q}_s\to F^{p,q}_{s+N-M}).$$

 \item \label{item:bijectMF2} If $N-M>0$, then $\PN^{N,M}$ is a bijective map from 
$$
Mult(F^{p,q}_s\to F^{p,q}_{s-N+M})\quad\text{to}\quad CF^{p,q}_s.
$$
\end{enumerate}
\end{thm}

Using this result, it can be obtained descriptions for $Mult(F_s^{p,q})$. In particular, we have that
$Mult(F_s^{p,q})=H^\infty \cap CF_s^{p,q}$ and we can answer our first question.

\begin{thm}\label{thm:bijectMH}
Let $1<p<\infty$ and $0< s\le n/p$. Then the following assertions are equivalent.

\begin{enumerate}
	\item\label{item:bijectMH1}$g\in Mult(H^{p}_s)$.
	\item\label{item:bijectMH2} $g\in H^\infty$ and there exists $h\in Mult(H^{p}_s\to H^p)$ such that 
$g=\mathcal{C}_s(h)$. 
That is, $g\in H^\infty$ and there exists $h\in   H^p$ such that 
$g=\mathcal{C}_s(h)$ and $|h|^pd\sigma$ is a Carleson measure for $H_s^p$ on $\bB$.
   \item\label{item:bijectMH4} $g\in H^\infty$ and $\PN^{0,s}(g)\in Mult(H^{p}_s\to H^p)$.
\end{enumerate}
\end{thm}

In the particular case that $0<n-sp<1$, we  show that in condition \eqref{item:bijectMH2}, we can substitute $h\in H^p$ by $h \in L^p(\bB)$. Also, it is worth to recall that, in this case, the Carleson measures $|h|^pd\sigma$ on $\bB$, can be characterized in terms of nonisotropic capacities.

A second focus of interest of this paper comes out from a unpublished work of \cite{boe} that deals with Bessel real potential spaces $L_{s,p}$. It is proved there that for the nonlinear potentials of positive measures, it is enough to impose the boundedness to assure that the function is a pointwise multiplier of $L_{s,p}$.

In our context of holomorphic functions and using completely different techniques, we describe different nontrivial examples of multipliers, which are summarized in the following theorems.

In the first one,    $\mathcal{I}_s$ denotes the nonisotropic Riesz operator given by
$$\mathcal{I}_s(\varphi)(z):=\int_\bB \frac{\varphi(\z)}{|1-z\overline \z|^{n-s}}d\sigma(\z).$$

\begin{thm}\label{prop:multiplierscasp=2}
Let $1<p<\infty$, $0<s<n/p$ and $\mu$ a finite positive Borel measure on $\bB$. We then have:
\begin{enumerate}
\item \label{item:multiplierscasp=21} 
 Assume that $\mathcal{I}_s(\mathcal{I}_s(\mu)^{p'-1})$ is bounded on $\bB$. Then ${\mathcal C}_s(\mathcal{I}_s(\mu)^{p'-1})$ is a multiplier for $H_s^p$.
\item \label{item:multiplierscasp=22}
Let $p=2$, and in addition, assume that $0<n-2s<1$. If  ${\mathcal C}_s({\mathcal C}_s(\mu))$ (respectively ${\mathcal C}_{2s}(\mu)$) is bounded, then ${\mathcal C}_s({\mathcal C}_s(\mu))$ (respectively  ${\mathcal C}_{2s}(\mu)$) is a multiplier for $H_s^2$.
\end{enumerate}
\end{thm}

 We also study the case  $p\neq 2$, and prove a similar result of the second statement of the above theorem, for the nonlinear potentials of positive measures introduced by \cite{cohnverbitsky} ${\mathcal U}_{s,p,\lambda}(\mu)$ and ${\mathcal V}_{s,p,\lambda}(\mu)$, which will be defined precisely in Section \ref{sec:multholpot}. In these cases, the boundedness is again enough to ensure the fact of being a pointwise multiplier for $H_s^p$. 
 
 \begin{thm}\label{thm:multpot}
Let $1<p<\infty$, $0<n-sp< \lambda<1$ and $\mu$ a finite positive Borel measure on $\bB$. Assume that either
the holomorphic potential ${\mathcal U}_{s,p,\lambda}(\mu)$ is bounded if $p< 2$ or
 the holomorphic potential ${\mathcal V}_{s,p,\lambda}(\mu)$ is bounded if $p\geq 2$.

We then have that if $p< 2$ the function ${\mathcal U}_{s,p,\lambda}(\mu)$ is a multiplier for $H_s^p$ and if $p\geq 2$ the function ${\mathcal V}_{s,p,\lambda}(\mu)$ is a multiplier for $H_s^p$.
\end{thm}
In particular, if the measure is the capacitary measure associated to compact subsets of $\bB$, which will be defined later, we have:

\begin{thm}\label{thm:multsextremal}
Let $1<p<\infty$, $0<s<n/p$  and let $E$ be a compact subset in $\bB$.   If $\mu_E$ is the capacitary extremal measure associated to $E$, we  have:
\begin{enumerate}
\item If $1<p<2$ and in addition $n-s<\lambda<1 $, the holomorphic potential ${\mathcal U}_{s,p,\lambda}(\mu_E)$ is a multiplier for $H_s^p$.
\item  If $p\geq 2$, and $n-sp<\lambda<1$, the holomorphic potential ${\mathcal V}_{s,p,\lambda}(\mu_E)$ is a multiplier for $H_s^p$.
\end{enumerate}
\end{thm}
 
Finally, we give some applications of the results obtained. We give a generalization of some examples of multipliers given by Beatrous and Burbea in \cite{beatrousburbea}, we solve a strong Corona problem for multipliers with data holomorphic potentials of capacitary measures, and we show that the sets of capacity zero are "weak exceptional sets" for the multipliers of $H_s^p$ in a sense that will be specified.

The paper is organized as follows:

 Section \ref{sec:preliminaries} is devoted to the preliminaries on Triebel-Lizorkin spaces and their space of pointwise multipliers. 
 In Section \ref{sec:operatorPNMF}  we study the action of the operators of Bergman type $\PN^{N,M}$ on Triebel-Lizorkin spaces. 
 The action of these operators on spaces of pointwise multipliers is studied in Section  \ref{sec:pointwise} where we also prove Theorem \ref{thm:bijectMF}.

 In Section \ref{sec:multHS} we consider multipliers of Hardy-Sobolev spaces and in particular we prove Theorem \ref{thm:bijectMH}.  
 
 Theorems  \ref{prop:multiplierscasp=2} and \ref{thm:multpot} are proved in Section \ref{sec:multholpot}. In Section \ref{sec:examples} we prove Theorem \ref{thm:multsextremal}.
Section \ref{sec:applications} is devoted to obtain  applications of the results obtained in the previous sections. Finally, in Section \ref{sec:mf}, we give the proof of the above mentioned generalized Leibnitz's formula.

\section{Preliminaries}\label{sec:preliminaries}

\subsection{Spaces of functions on $\B$}

In this section we recall some properties of the holomorphic Besov and Hardy-Sobolev spaces on $\B$.

\begin{defn}
Let $1\le p<\infty$ and $s\in\R$. 
The holomorphic Besov space $B^p_s$ consists of holomorphic functions on $\B$ such that 
$$
\|f\|_{B^p_s}^p:=\int_B |(1+R)^{k}f(z)|^p(1-|z|^2)^{(k-s)p-1}d\nu(z)<\infty,
$$ 
for a nonnegative integer $k>s$. 

If $p=\infty$, then $B^\infty_s$ consists of the holomorphic functions on $\B$ such that for a nonnegative integer $k>s$,
$$
\|f\|_{B^\infty_s}:=\sup_{z\in\B}|(1+R)^{k}f(z)|(1-|z|^2)^{k-s}<\infty.
$$ 
\end{defn}

It is well known that for $p\leq \infty$, different values of $k$ give equivalent norms on $B^p_s$.

The spaces $H^p_s$ and $B^p_s$ are in the scale of Triebel-Lizorkin spaces.
Let us recall the definition of these spaces.

\begin{defn}
For $\,\z\in \bB,\,$ let $\,\Gamma(\z)\,$ be
the admissible region
$$ 
\Gamma(\z)=\left\{ z \in \B:\,|1-z\bar \z |<2(1-|z|^2)\,\right\}.\,
$$

Let $\mu$ be a positive Borel measure on $\B$.
We denote by $T^{p,q}(\mu)$ the non-isotropic tent space of measurable functions $\varphi$ on $\B$ such that 
$$
\|\varphi\|_{T^{p,q}(\mu)}^p :=\int_{\bB}\left(\int_{\Gamma_\z}|\varphi(w)|^q\,\frac{d\mu(w)}{(1-|w|^2)^{n+1}}\right)^{p/q} d\sigma(\z)
$$
is finite.

If $\mu=\nu$, then we write  $T^{p,q}$.
\end{defn}

\begin{defn} 
Let $1\le p,q<\infty$ and $s\in\R$.
The holomorphic Triebel-Lizorkin space $F^{p,q}_s$ consists of holomorphic functions on $\B$ such that 
$$
\|f\|_{F^{p,q}_s}:=\left\|(1-|z|^2)^{k-s}(1+R)^{k} f(z)\right\|_{T^{p,q}}<\infty,
$$
for a  nonnegative integer $k>s$.
\end{defn}

As it happens in other spaces of holomorphic functions,  
 different values of $k$ provide equivalent norms on $F^{p,q}_s$.

The next proposition gives some embeddings between these spaces. For functions in the corresponding real spaces, the proof of such embeddings can be found in \cite{Tr}.
A proof of the embeddings between  Hardy-Sobolev spaces  and  holomorphic Besov spaces  
using techniques of complex variables can be found in \cite{beatrousburbea}. 
Extensions of these results to Triebel-Lizorkin spaces can be found 
for instance in \cite[Theorem 4.1]{Or-Fa2}. 

\begin{prop}\label{prop:embedF}
Let $1\le p,q<\infty$ and $s,t\in\R$. Then:  
\begin{enumerate}
	\item $H^p_s=F^{p,2}_s$ and  $B^p_s=F^{p,p}_s$.
	
\item For any $\e>0$, we have
$$
B^\infty_{s+n/p'+\e}\subset B^1_{s+n/p'}\subset F^{p,q}_s\subset B^\infty_{s-n/p}\subset B^1_{s-n/p-\e}.
$$
\item If $q_0>q$, then $F^{p,q}_s\subset F^{p,q_0}_s$.
\item\label{item:embedF4} If $p_0>p$ and $s_0=s+n/p_0-n/p$, then $F^{p,q}_s\subset F^{p_0,1}_{s_0}\subset H^{p_0}_{s_0}$.
	\item\label{item:embedF5} If $t<s$, then $F^{p,q}_s\subset B^1_{t}$.
\end{enumerate}
\end{prop}

The next result can be found, for instance, in \cite{ortegafabrega} and \cite{Or-Fa3}.
\begin{thm}
Let $1<p<\infty$ and $s>0$. Then, $Mult(H^p_s)=H^\infty\cap CF^{p,2}_s$ and $Mult(B^p_s)=H^\infty\cap CF^{p,p}_s$.
\end{thm}

 \subsection{Properties of the space  $ Mult(F^{p,q}_s\to F^{p,q}_{s'})$}\quad\par

In this section we summarize some properties of the space of pointwise multipliers  
$ Mult(F^{p,q}_s\to F^{p,q}_{s'})$. Some of them are well known, 
specially the ones corresponding to Hardy-Sobolev spaces ($q=2$ with $s=s'$) or 
Besov spaces $(q=p)$ with $s=s'$. Since the possible references  are scattered, and not always in the generality that we require, we have thought proper to give a sketch of the proof of all them, using several usual techniques.

\begin{prop}\label{prop:grown}
Let $1< p,q<\infty$ and $s,s'\in\R$.  Then, we have:   

\begin{enumerate}
	\item \label{item:grown1} If $s'<s\le n/p$, then $ Mult(F^{p,q}_s\to F^{p,q}_{s'})\subset B^\infty_{s'-s}\cap F^{p,q}_{s'}$. If $s'<s<0$, then $ Mult(F^{p,q}_s\to F^{p,q}_{s'})= B^\infty_{s'-s}$
	\item \label{item:grown2}If  $s\le n/p$, then $ Mult(F^{p,q}_s)\subset H^\infty\cap F^{p,q}_{s}$. If $s<0$, then $ Mult(F^{p,q}_s)= H^\infty$.
	\item \label{item:grown7} If $s'\le s$ and $s>n/p$,  then $Mult(F^{p,q}_s\to F^{p,q}_{s'})=F^{p,q}_{s'}$. In particular, $Mult(F^{p,q}_s)$ is a multiplicative algebra.
	\item \label{item:grown3} If $k>s\ge s'$, then $ B^\infty_k\subset Mult(F^{p,q}_s\to F^{p,q}_{s'})$.
	\item \label{item:grown4} If $s'\le s
	$ and $\tau>0$, then 
	$$
	Mult(F^{p,q}_{s}\to F^{p,q}_{s'})\subset Mult(F^{p,q}_{s-\tau}\to F^{p,q}_{s'-\tau}).
	$$
 	\item \label{item:grown5} If $s'\le s$ and $g\in Mult(F^{p,q}_s\to F^{p,q}_{s'})$, then 
$g$ and $(1+R)^k g\in Mult(F^{p,q}_s\to F^{p,q}_{s'-k})$ for any positive integer $k$.
	\item \label{item:grown6} If $s'>s$, then $Mult(F^{p,q}_s\to F^{p,q}_{s'})=\{0\}$.

\end{enumerate}
\end{prop}

\begin{rem}
In Proposition \ref{prop:charMssp} we  will prove the following result which completes the above statement \eqref{item:grown5}:

\begin{enumerate}\setcounter{enumi}{7}
	\item  If $s'<s$ and $(1+R)^k g\in Mult(F^{p,q}_s\to F^{p,q}_{s'-k})$ for some positive integer $k$, then $g\in Mult(F^{p,q}_s\to F^{p,q}_{s'})$.  
\end{enumerate}
\end{rem}

\begin{proof}
Let us prove \eqref{item:grown1}. Assume $s'<s\le n/p$.
Since $F^{p,q}_s$ contains the constant functions, then it is clear that 
$Mult(F^{p,q}_s\to F^{p,q}_{s'})\subset  F^{p,q}_{s'}$.

In order to estimate the growth of $g\in Mult(F^{p,q}_s\to F^{p,q}_{s'})$, 
we use that $g\in Mult(B^1_{s+n/p'}\to B^\infty_{s'-n/p})$, which follows 
from the embeddings $B^1_{s+n/p'}\subset F^{p,q}_s$ and 
$F^{p,q}_{s'}\subset B^\infty_{s'-n/p}$ stated in Proposition \ref{prop:embedF}.

Fixed $z\in\B$, let $f_z(w)=\frac{(1-|z|^2)^{n+N}}{(1-w\overline z)^{n+N}}$.  By Proposition 1.4.10 in \cite{rudin} (see also Lemma \ref{lem:estP} below)
for  $N$ large enough 
we have that $\|f_z\|_{B^1_{s+n/p'}}\approx (1-|z|^2)^{n/p-s}$.
 Thus,  
\begin{equation}\label{eqn:grown1}\begin{split}
&(1-|z|^2)^{n/p-s'}|g(z)|\\
&\le \|(1-|w|^2)^{n/p-s'}g(w)f_z(w)\|_{L^\infty}
\approx \|g f_z\|_{B^\infty_{s'-n/p}}\\
&\lesssim  \|g\|_{Mult(B^1_{s+n/p'}\to B^\infty_{s'-n/p})}(1-|z|^2)^{n/p-s},
\end{split}\end{equation}
which gives $|g(z)|\lesssim \|g\|_{Mult(F_s^{p,q}\to F_{s'}^{p,q})}(1-|z|^2)^{s'-s}$. 
Hence $\|g\|_{B^\infty_{s'-s}}\lesssim \|g\|_{Mult(F_s^{p,q}\to F_{s'}^{p,q})}$.

If, now,  $s'<s<0$ and $g\in B^\infty_{s'-s}$, we have  
\begin{equation}\label{eqn:grown2}\begin{split}
\|gf\|_{F^{p,q}_{s'}}&\approx\|g(z)f(z)(1-|z|^2)^{-s'}\|_{T^{p,q}}\\
&\lesssim \|g\|_{B^\infty_{s'-s}}\|f(z)(1-|z|^2)^{-s}\|_{T^{p,q}}\\
&\approx \|g\|_{B^\infty_{s'-s}}\|f\|_{F^{p,q}_{s}},
\end{split}\end{equation}
which gives $\|g\|_{Mult(F_s^{p,q}\to F_{s'}^{p,q})}\lesssim \|g\|_{B^\infty_{s'-s}}$.

In order to prove \eqref{item:grown2},  let $g\in Mult(F^{p,q}_s)$. For any $f\in F^{p,q}_s$ and $z\in\B$, then
$$
|(gf)(z)|\le \sup_{\substack{h\in F^{p,q}_s\\h\ne 0}}\frac{|h(z)|}{\|h\|_{F^{p,q}_s}}\|gf\|_{F^{p,q}_s}
\le \sup_{\substack{h\in F^{p,q}_s\\h\ne 0}}\frac{|h(z)|}{\|h\|_{F^{p,q}_s}} \|g\|_{Mult(F^{p,q}_s)}\|f\|_{F^{p,q}_s},
$$
which gives $\|g\|_\infty\le \|g\|_{Mult(F^{p,q}_s)}$.

The same arguments  used to prove \eqref{eqn:grown2} show that if $s<0$, 
then the converse inequality holds. Thus \eqref{item:grown2} is proved.

In order to prove \eqref{item:grown7},  we first consider the case $s'=s>n/p$, and we must show that $F^{p,q}_s$ is a multiplicative algebra.

For a nonnegative integer $j$, $t\ge 0$ and $h\in H$, 
let $D_{j,t}(h)(z):=(1-|z|^2)^{j-t}(1+R)^j h(z)$. 
Since $F^{p,q}_s\subset B^\infty_{s-n/p}\subset H^\infty\subset B^\infty_0$, Leibnitz's formula for $k>s$ gives 
\begin{align*}
&\|fg\|_{F^{p,q}_s}\approx \|(1-|z|^2)^{2k-s}(2+R)^{2k}(fg)(z)\|_{T^{p,q}}\\
&\lesssim\sum_{j=0}^{k} \|D_{j,0}(f)\|_{L^\infty} \|D_{2k-j,s}g\|_{T^{p,q}}
+\sum_{j=k+1}^{2k} \|D_{2k-j,0}g\|_{L^\infty} \|D_{j,s} (f)\|_{T^{p,q}}\\
&\lesssim \|f\|_{H^\infty}\|g\|_{F^{p,q}_s}+
\|g\|_{H^\infty}\|f\|_{F^{p,q}_s}\lesssim \|f\|_{F^{p,q}_s}\|g\|_{F^{p,q}_s},
\end{align*}
which proves that $Mult(F^{p,q}_s)=F^{p,q}_s$.

Let $s'<s$ and let $s_0<0$ and $s_0<s'$. The Calderon complex method in the theory of interpolation spaces gives  
$F^{p,q}_{s'}=(F^{p,q}_s,F^{p,q}_{s_0})_{[\theta]}$ for $\theta$ satisfying $s'=(1-\theta)s+\theta s_0$
(see for instance \cite[Corollary 3.4]{Or-Fa1}). 

Thus, since for $s>n/p$, we have $F^{p,q}_s\subset H^\infty=Mult(F^{p,q}_{s_0})$ 
and $F^{p,q}_s\subset Mult(F^{p,q}_{s'})$, which proves  
$F^{p,q}_{s'}\subset Mult(F^{p,q}_{s}\to F^{p,q}_{s'})$ which finishes the proof of \eqref{item:grown7}.

Observe that if $s>n/p$, then $F^{p,q}_{s'}\subset B^\infty_{s'-n/p}\subset B^\infty_{s'-s}$. Therefore assertion \eqref{item:grown1} is also satisfied for $s>n/p$. The same argument holds for \eqref{item:grown2}.

Assertion \eqref{item:grown3} follows easily from Leibnitz's formula
\begin{equation}\label{eqn:Leibnitz}
(1+R)^k(gf)=\sum_{j=0}^k c_j (1+R)^j g R^{k-j}f\,
\end{equation}
and the fact that if $g\in B^\infty_k$, 
then $(1+R)^jg\in H^\infty$ for $0\le j<k$ and $(1+R)^k g\in B^\infty_0$.
 
In order to prove \eqref{item:grown4}, note that assertions 
\eqref{item:grown1}, \eqref{item:grown2} and \eqref{item:grown7} 
show that if $g \in Mult(F^{p,q}_s\to F^{p,q}_{s'})$,  
then $g\in Mult(F^{p,q}_{-s_0}\to F^{p,q}_{s'-s-s_0})$ for any $s_0>0$. Therefore \eqref{item:grown4} follows from the above mentioned interpolation result.

Consider now \eqref{item:grown5}. Let  $g\in Mult(F^{p,q}_s\to F^{p,q}_{s'})$. 
Since $F^{p,q}_{s'}\subset F^{p,q}_{s'-k}$, it is clear that $g\in Mult(F^{p,q}_s\to F^{p,q}_{s'-k})$.
By \eqref{item:grown4}, we also have  
$g\in  Mult(F^{p,q}_{s-1}\to F^{p,q}_{s'-1})$. Thus,
if $f\in F^{p,q}_{s'}$, then the identity $f(1+R)g=(1+R)(fg)-gRf$ gives 
 $f(1+R)g\in F^{p,q}_{s'-1}$, which proves that  
 $(1+R)g\in Mult(F^{p,q}_{s}\to F^{p,q}_{s'-1})$. Iterating this result we obtain \eqref{item:grown5}.  

To conclude we prove \eqref{item:grown6}. Consider first the case $s'<n/p$. If 
$g\in Mult(F^{p,q}_s\to F^{p,q}_{s'})$, $s'>s$, then by \eqref{eqn:grown1}, 
$(1-|z|^2)^{s-s'}|g(z)|$ is bounded, and consequently $g=0$. 
 
The proof for the case $s'\geq n/p$ can be reduced to this case. Since $F^{p,q}_{s'}\subset F^{p,q}_s$, we have 
$$
Mult(F^{p,q}_s\to F^{p,q}_{s'})\subset Mult(F^{p,q}_{s'})\subset H^\infty=Mult(F^{p,q}_{-s'}\to F^{p,q}_{-s'}).
$$
Therefore, by interpolation with $\theta=1/2$ we obtain 
$$
Mult(F^{p,q}_s\to F^{p,q}_{s'})\subset Mult(F^{p,q}_{(s-s')/2}\to F^{p,q}_{0})=\{0\}
$$
which concludes the proof.
\end{proof}

\begin{rem}
Proposition \ref{prop:grown} shows that the space 
$Mult(HF^{p,q}_s\to HF^{p,q}_{s'})$  coincides with $B_{s'-s}^\infty$ if $s'<s<0$ and with $H^\infty$ if $s=s'<0$. On the other hand, if $s'\leq s$  and $s>n/p$,  it coincides with $F_{s'}^{p,q}$,  and is identically zero if $s'>s$.
In all these situations the space of multipliers has a simple description. 
Therefore, in the rest of the paper we only consider the remainder case  $0\le s\le n/p$ and $s'\le s$.
\end{rem}

\begin{cor}\label{cor:charMFs}
Let $1<p,q<\infty$ and $0\le s\le n/p$. Then $g\in Mult(F^{p,q}_s)$ if and only if 
$g\in H^\infty$ and for some (any) nonnegative integer $k>s$,  
$(1+R)^j g\in Mult(F^{p,q}_{s}\to F^{p,q}_{s-j})$, for $j=1,\cdots,k$.
\end{cor}

\begin{proof}
By Proposition \ref{prop:grown}\eqref{item:grown5}, it is clear that if $g\in Mult(F^{p,q}_s)$ then the conditions are satisfied. The converse follows from the Leibnitz's formula \eqref{eqn:Leibnitz} 
and the fact that, by Proposition \ref{prop:grown}\eqref{item:grown4}, 
$(1+R)^j g\in Mult(F^{p,q}_{s+j-k}\to F^{p,q}_{s-k})$.
\end{proof}

In Section \ref{sec:pointwise} we will obtain some characterizations of $Mult(F^{p,q}_s)$, 
which in particular give that $g\in Mult(F^{p,q}_s)$ if and only if $g\in H^\infty$ 
and for some (any) nonnegative integer $k$, $(1+R)^k g\in Mult(F^{p,q}_{s}\to F^{p,q}_{s-k})$. 
This fact is well known for Hardy-Sobolev and Besov spaces.

\section{The operator $\PN^{N,M}$ on $F^{p,q}_s$}\label{sec:operatorPNMF}

\subsection{Differential and integral operators}

For $N>0$, we denote by $d\nu_N$ the probabilistic measure 
$$
d\nu_N(z):=c_N(1-|z|^2)^{N-1}d\nu(z), \qquad 
c_1=1, \qquad c_{N+1}=\frac{n+N}N c_N
$$
In order to unify the statements, for $N=0$ we define 
$d\nu_0=d\sigma$.

\begin{defn}
For $N>0$ and $M>-n$, we consider the following integral operators:
$$
\PN^{N,M}(\varphi)(z):=\int_\B \varphi(w) \PN^{N,M}(z,w)d\nu(w),\quad \PN^{N,M}(z,w):=c_N\frac{(1-|w|^2)^{N-1}}{(1-z\overline w)^{n+M}}.
$$

$$
\PP^{N,M}(\varphi)(z):=\int_\B \varphi(w)\PP^{N,M}(z,w)d\nu(w),\quad \PP^{N,M}(z,w):=|\PN^{N,M}(z,w)|.
$$

We extend the definition to the case $N=0$ by 
$$
\PN^{0,M}(\varphi)(z):=\int_\bB \frac{\varphi(\z)}{(1-z\overline \z)^{n+M}}d\sigma(\z),\quad \PP^{0,M}(\varphi)(z):=\int_\bB \frac{\varphi(\z)}{|1-z\overline \z|^{n+M}}d\sigma(\z).
$$

If $N=M$, then we denote $\PN^{N,N}$ and $\PP^{N,N}$ by $\PN^N$ and  $\PP^N$, respectively. 
\end{defn}

Observe that  ${\mathcal C}_s=\PN^{0,-s}$ and ${\mathcal I}_s=\PP^{0,-s}$.

Throughout the paper, the operators $\PN^{N,M}$ considered satisfy that $N\geq 0$, $M>-n$. Consequently, in general, we will not include these hypotheses in the statement of our results, except on the cases, where the range has to be restricted.

The next representation formula is well known.
\begin{prop}\label{prop:repPNN}
If either $f\in B^1_{-N}$ for $N>0$ or $f\in H^1$ for $N=0$, then $f=\PN^{N}(f)$.
\end{prop}

It will be useful to introduce the following differential operators, 
which play the same role than  $(1+R)^k$, but allow to simplify some computations.

\begin{defn}
Let $k$ be a positive integer and let $t>0$. We denote by  $R^k_t$ the differential operator of order $k$ defined by 
$$
R^k_{t}:=\left(1+\frac{R}{t+k-1}\right)\cdots\left(1+\frac{R}{t}\right).
$$

If $k=0$, then $R^0_{t}$ denotes the identity operator.
\end{defn}

These operators satisfy the following properties:
\begin{lem} \label{lem:Rkt}
Let $k,m$ be nonnegative integers and let $t>0$. Then, we have:
\begin{enumerate}
	\item $R^m_{t+k}\circ R^k_t=R^{k+m}_t$.
	\item If $n+M>0$, then 
\begin{equation*}
R^{k}_{n+M,z}\frac{1}{(1-z\overline w)^{n+M}}=\frac{1}{(1-z\overline w)^{n+M+k}}.
\end{equation*}
	\item \label{eqn:Rkt3} If $f\in B^1_{-N}$, then $R^k_{n+N}f(z)=R^k_{n+N}\PN^Nf(z)=\PN^{N,N+k}(f)$.
	\item If $1\le p,q<\infty$, then the operator $R^k_t$ is a bijective operator from $F^{p,q}_s$ to $F^{p,q}_{s-k}$
\end{enumerate}
\end{lem}

Observe that 
from \eqref{eqn:Rkt3} 
the integral operator $\PN^{N,N+k}$ acts on $B^1_{-N}$ as a differential operator of order $k$.

The next result will be used to obtain norm-estimates of the operators $\PN^{N,M}$ and $\PP^{N,M}$.

\begin{lem}\label{lem:estP}  Let $N>0$ and $M,L\ge -n$. Then the integral 
$$
I^N_{M,L}(z,u):=\int_B\frac{(1-|w|^2)^{N-1}}{|1-u\overline w|^{n+M}|1-z\overline w|^{n+L}}d\nu(w)
$$
satisfies the following estimates:

If $M>N>L\ge -n$, then
$\displaystyle{
I^N_{M,L}(z,u)\lesssim 
\frac{(1-|u|^2)^{N-M}}{|1-z\overline u|^{n+L}}.
}$

If $M,L>N$, then
$\displaystyle{
I^N_{M,L}(z,u)
\lesssim \frac{(1-|u|^2)^{N-M}}{|1-z\overline u|^{n+L}}+\frac{(1-|z|^2)^{N-L}}{|1-z\overline u|^{n+M}}.
}$

If $M,L<N\ne n+M+L$, then
$\displaystyle{
I^N_{M,L}(z,u)
\lesssim 1+\frac{1}{|1-z\overline u|^{n+M+L-N}}.
}$

If $M,L<N=n+M+L$, then 
$\displaystyle{
I^N_{M,L}(z,u)
\lesssim \log \frac{e}{|1-z\overline u|}.
}$
\end{lem}
 
 If $u=0$, then the above estimates corresponds to Proposition 1.4.10 in \cite{rudin}. 
 The proof of the  estimates in the lemma can be deduced easily from this case by 
 decomposing the integral on $\B$ in two parts, one over the set $\Omega_1=\{w\in\B: |1-z\overline w|\le |1-u\overline w|\}$ and the other over $\Omega_2=\B\setminus \Omega_1$. 
  Using that $|1-z\overline u|+|1-z\overline w|\approx |1-u\overline w|$  for $w\in\Omega_1$, and the analogous equivalence for $w\in \Omega_2$ we obtain the estimates.
 
The next lemma was proved in \cite[Proposition 2.8]{Or-Fa1}.
\begin{lem} \label{lem:PPF}
Let $1< p<\infty$, $1\le q<\infty$ and $N,L>0$. If $N+L\ge M$, then the operator
$
\psi\to (1-|z|^2)^{L}\PP^{N,M}(\psi)(z)
$ 
is bounded on $T^{p,q}$.
\end{lem}
 
\begin{prop} \label{prop:intpartsphi}
Let $\varphi\in \mathcal{C}^k(\overline\B)$. Then, for any $N\ge 0$ we have:
$$
\int_\B \varphi d\nu_N=\int_\B R^k_{n+N}\varphi d\nu_{N+k} =\int_\B \overline{R^k_{n+N}}\varphi d\nu_{N+k}.
$$

In particular, if $\varphi\in \mathcal{C}^{k+m}(\overline\B)$, then 
\begin{equation}\label{eqn:intpartsphi}
\int_\B \varphi d\nu_N=\int_\B \overline{R^m_{n+N+k}}R^k_{n+N}\varphi d\nu_{N+k+m}. 
\end{equation}
\end{prop}

\begin{proof}
Note that in order to prove  the first equality it is enough to consider $k=1$ and reiterate the formula.  The second identity follows  conjugating the first one. 

Assume $N>0$.
Since $c_{N+1}=\frac{n+N}{N}c_N$, $$c_N(1-|w|^2)^{N-1}=\frac{N}{n+N}c_{N+1}(1-|w|^2)^N-\frac{c_{N+1}}{n+N} R(1-|w|^2)^N,$$ by integration by parts we have
\begin{align*}
\int_\B\varphi d\nu_N
&=\frac{N}{n+N} \int_\B \varphi d\nu_{N+1} 
+\frac{1}{n+N}\sum_{j=1}^n\int_\B \p_j(w_j\varphi(w))d\nu_{N+1}(w)\\
&=\int_\B \varphi d\nu_{N+1}+\frac{1}{n+N}\int_\B R\varphi d\nu_{N+1}.
\end{align*}
This proves  the case $N>0$.

The case $N=0$ follows from the Stokes' theorem. In this case there exists a constant $a_n$ such that 
$d\sigma(\z)=a_n \omega(\z)$ where 
$$
\omega(\z)=\sum_{j=1}^{n} (-1)^{j-1}\z_j d\z_1\land \cdots d\z_{j-1}\land d\z_{j+1}\land \cdots d\z_{n} 
\land d\overline{\z}_1\land \cdots d\overline{\z}_n,
$$
and thus 
\begin{align*}
\int_{\bB}\varphi d\sigma =a_n\int_{\B} d(\varphi\omega)
=b_n \int_{\B}\left(1+\frac{R}{n}\right)\varphi d\nu.
\end{align*}
If in this result we consider the constant function $\varphi=1$, we obtain $b_n=1$, which ends the proof.
\end{proof}

\begin{cor}\label{cor:intpartsP}
Let $f\in B^1_{-N}$ if $N>0$ and $f\in H^1$ if $N=0$.

Then, for any nonnegative integers $k,m$, we have
\begin{equation}\label{eqn:intpartsP0}
\PN^{N,M}(f)(z)=R^m_{n+N+k}\left[\PN^{N+k+m,M} (R^k_{n+N}f)\right](z).
\end{equation}
\end{cor}
\begin{proof}
Applying \eqref{eqn:intpartsphi} to the function $\varphi_z(w)=\dfrac{f(w)}{(1-z\overline w)^{n+M}}$, $z\in\B$, and using the fact that $\overline{R^m_{n+N+k,w}}(1-z\overline w)^{-M-n}=R^m_{n+N+k,z}(1-z\overline w)^{-M-n}$ we obtain the result.
\end{proof}

\begin{cor}\label{cor:RkvPN}
Let  $N,M>0$ and  $h\in B^\infty_{-N}$.
Then for any positive integer $m$, there exist constants $a_j$, $j=1,\dots,m+1$, such that 
\begin{equation}\label{eqn:RkvPN1}
h= \sum_{j=1}^m a_j\PN^{N+j,N}(R^1_M h)+a_{m+1}\PN^{N+m,N}(h).
\end{equation}
\end{cor}

\begin{proof}
By Corollary \ref{cor:intpartsP}, we have
\begin{align*}
&\PN^{N+j,N}(h)=\PN^{N+j+1,N}(R^1_{n+N+j}h)\\
&=\frac{M}{n+N+j}\PN^{N+j+1,N}\left(\left(1+\frac{R}{M}\right)h\right)+\frac{n+N+j-M}{n+N+j}\PN^{N+j+1,N}(h).
\end{align*}

Thus, if $j=0$, then $\PN^{N,N}(h)=h$, and we obtain the case $m=1$.

Assume that the result is valid for $m-1$. The above formula with $j=m-1$ gives 
$$
\PN^{N+m-1,N}(h)
=b_1\PN^{N+m,N}(R^1_Mh)+b_2\PN^{N+m,N}(h).
$$
By induction, this permits to complete the proof.
\end{proof}

\begin{prop}\label{prop:PNMonF}
Let $1<p,q<\infty$. Assume that either $s>-N$ or $s=N=0$ and $q\le 2$. 
Then, the operator $\PN^{N,M}$ is a bounded operator from 
$F^{p,q}_s$ to $F^{p,q}_{s+N-M}$.
\end{prop}

\begin{proof}
The condition $s>-N$ ensures that $F^{p,q}_s\subset B^1_{-N}$, and the condition $s= N=0$ and $q\le 2$
 ensures that $F^{p,q}_s\subset H^p\subset H^1$. Therefore, $\PN^{N,M}$ is well defined on $F^{p,q}_s$.

As a consequence of \eqref{eqn:intpartsP0} in Corollary \ref{cor:intpartsP}, 
for positive integers $l,k,m$ satisfying $l>s+N-M$ and $k>s$, we have
\begin{equation}\label{eqn:intpartsP}
|(1+R)^l\PN^{N,M}(f)(z)|\lesssim \PP^{N+k+m,M+l+m} (|R^k_{n+N}f|)(z).
\end{equation}
Therefore, 
\begin{align*}
\Psi(z):&=(1-|z|^2)^{l-s-N+M}|(1+R)^l\PN^{N,M}(f)(z)|\\
&\lesssim (1-|z|^2)^{l-s-N+M}\PP^{N+s+m,M+l+m} ((1-|w|^2)^{k-s}|R^k_{n+N}f(w)|)(z).
\end{align*}

Thus, by Lemma \ref{lem:PPF} we have 
$$
\|\PN^{N,M}(f)\|_{F^{p,q}_{s+N-M}}\approx \|\Psi\|_{T^{p,q}}
\lesssim  \|(1-|z|^2)^{k-s}R^k_{n+N}f(z)\|_{T^{p,q}}\approx \|f\|_{F^{p,q}_s},
$$
which ends the proof.
\end{proof}

Analogously, as a consequence of \eqref{eqn:intpartsP}, Fubini's theorem and Lemma \ref{lem:estP}, we obtain the following result.

\begin{prop}\label{prop:PNMonB}
The operator $\PN^{N,M}$ is  bounded  from 
$B^1_{-N}$ to $B^1_{-M}$ and from $B^\infty_{t}$ to $B^\infty_{t+N-M}$ for $t>-N$.
\end{prop}

\subsection{The inverse of the operator $\PN^{N,M}$} \label{subsec:invPNM}\quad\par

Since $\overline{\PN^M}$ is a reproducing kernel 
for smooth  anti-holomorphic functions on $\B$, for $z,w\in \B$ we have
\begin{equation} \label{eqn:intPNM}\begin{split}
\int_\B \PN^{M,N}(z,u)\PN^{N,M}(u,w)d\nu(u)&=\PN^N(z,w),\quad N\ge 0, M>0
\end{split}\end{equation}

Assume first that both $N, M>0$. If $f\in B^1_{-t}$, $t>-N$, then Fubini's theorem gives
$\PN^{M,N}[(P^{N,M}(f)]=f$.  

Therefore, 
$\PN^{N,M}$ is a bijective operator from $B^1_{-t}$ to $B^1_{-t+N-M}$ and its inverse is $\PN^{M,N}$.
This result together Proposition \ref{prop:PNMonF} proves that if $s>-N$, then $\PN^{N,M}$ 
is also bijective from $F^{p,q}_s$ to $F^{p,q}_{s+N-M}$.

Next, we consider the general case where $N\ge 0$ and $M>-n$. By \eqref{eqn:intpartsP0} with $k=0$ and $m>\max\{0,-M\}$,  
 we have 
\begin{align*}
\PN^{N,M}(f)(z)&=R^{m}_{n+N}\left[\PN^{N+m,M} (f)\right](z)\\
&=R^{m}_{n+N}(R^{m}_{n+M})^{-1} \left[\PN^{N+m,M+m} (f)\right](z).
\end{align*}
Since for $t>-N-m$, $\PN^{N+m,M+m}$ is a bijective operator from $B^1_t$ to $B^1_{t+N-M}$ and 
$R^{m}_{n+N}(R^{m}_{n+M})^{-1}$ is a bijective operator on $B^1_{t+N-M}$, then $\PN^{N,M}$ 
is a bijective operator from $B^1_{-N}$ to $B^1_{-M}$, and $\PN^{0,M}$ is a bijective operator 
from $H^p$ to $F^{p,2}_{-M}$, $1<p<\infty$.

Clearly from the above formula we can describe the inverse of $\PN^{N,M}$. However, for our purposes we will give another expression of this inverse operator.

From \eqref{eqn:intPNM} and Proposition \ref{prop:intpartsphi}, for any $l>\max\{0,-M\}$ we have
\begin{align*}
f(z)&=\PN^{M+l,N}\left[\PN^{N,M+l}(f)\right](z)
=\int_{\B}\frac{R^l_{n+M}\PN^{N,M}(f)(w)}{(1-z\overline w)^{n+N}}d\nu_{M+l}(w)\\
 &=\int_{\B}\PN^{N,M}(f)(w)\overline{R^l_{n+M,w}}(1-z\overline w)^{-n-N}d\nu_{M+l}(w)\\
 &=R^{l}_{n+M}\left[\PN^{M+l,N}\left(\PN^{N,M}(f)\right)\right](z).
\end{align*}

Combining the above results with the ones of Propositions \ref{prop:PNMonF} and  \ref{prop:PNMonB}, 
we obtain the inverse of $\PN^{N,M}$ as a linear combination of operator of the same type. 
All this is summarized in the following proposition. 

\begin{prop} \label{prop:PNMbij}
Let $1< p,q<\infty$. Then we have:
\begin{enumerate}
		\item If $N>0$ and $M>-n$, then the operator $\PN^{N,M}$ is a bijective operator from $B^1_{-N}$ to $B^1_{-M}$. Its inverse is given by an integral operator 
	$$
	\T^{N,M}(f)(z):=\int_{\B} \T^{N,M}(z,w)f(w)d\nu(w),
	$$
	where for a positive integer $l>\max\{0,-M\}$, 
	\begin{equation}\label{eqn:expT}
	 \T^{N,M}(z,w):=R^{l}_{n+M}\PN^{M+l,N}(z,w)=\sum_{j=0}^{l} a_j \PN^{M+l,N+j}(z,w).
	\end{equation}
	
	\item If $N\ge 0$ and either $s>-N$ or $N=s=0$ and $q\le 2$, then $\PN^{N,M}$ is a bijective operator from $F^{p,q}_{s}$ to $F^{p,q}_{s+N-M}$.
	
	\item If $N\ge 0$ and $t>-N$, then $\PN^{N,M}$ is a bijective operator from $B^\infty_t$ to $B^\infty_{t+N-M}$.
\end{enumerate}
\end{prop}

\section{Pointwise multiplier norm-estimates of $\PN^{N,M}º$ }\label{sec:pointwise}

In Proposition \ref{prop:PNMbij}, we proved the bijectivity of the operator $\PN^{N,M}$  from $F^{p,q}_s$ to $F^{p,q}_{s+N-M}$. 
In this section we prove that, under some natural assumptions on $N,M,s$ and $s'$,  this operator restricted to $Mult(F^{p,q}_s\to F^{p,q}_{s'})$ is also  bijective  from this space to $Mult(F^{p,q}_s\to F^{p,q}_{s'+N-M})$.

\subsection{The operator $\PN^{N,M}$ on spaces of pointwise multipliers}\qquad\par

We begin studying the boundedness of the operator $\PN^{N,M}$ 
on some spaces of pointwise multipliers. The following formula will be the essential tool in order to prove this result. 
It involves constants $k,J,L$ and $m$ satisfying some conditions, 
which always exist because it is enough to choose them big enough.

\begin{thm}\label{thm:mf}
Let $N\ge 0$, $0<t<n+N$ and $M>-n$.  If $N>0$, let $g\in B^1_{-N}\cap B^\infty_{-t}$,  and if $N=0$, let $g\in H^1(\B)\cap B^\infty_{-t}$. 
Let also  $k,J,L$ and $m$ be positive integers satisfying: 
$$
k>M+n-1,\quad J>k+n+N,\quad  L>N+J,\quad \text{and} \quad m>N-t+n+J.
$$ 
Then,  there exist constants $a_i$ and $a_{i,j}$, 
such that  for any  $f\in B^1_{t+k-N-J}$,
\begin{equation}\label{eqn:mf}\begin{split}
f\PN^{N,M}(g)&= \sum_{0\le i \le k+1-n-M} a_{i} f\PN^{N+J,M+i}(g)\\
&+\sum_{k+1-n-M<i\le J}\sum_{j=0}^k a_{i,j}  \PN^{N+J,M+i-j}\left(g\,d^j\,f(R,\overset{(j)}{\cdots},R)\right)\\
&+ Q^{N,M,k}(f,g).
\end{split}\end{equation}
The function $ Q^{N,M,k}(f,g) $ is a holomorphic function on $\B$  satisfying 
\begin{align*}
&|(1+R)^m  Q^{N,M,k}(f,g)(z)|\\ 
&\lesssim 
\|g\|_{B^\infty_{-t}}\left(\int_\B |R^l_{n+L} f(u)|
\frac{(1-|u|^2)^{N-t-k+J+l-1}}{|1-z\overline u|^{n+M-k+J+m}} d\nu(u)\right.\\
&\left.+(1-|z|^2)^{k+1-m}\Omega_{N-t-M}(1-|z|^2)
\int_\B |R^l_{n+L} f(u)|\frac{(1-|u|^2)^{L+l-1}}{|1-z\overline u|^{n+k+L+1}}d\nu(u)\right),
\end{align*}
where $l$ is nonnegative integer and  the funtion $\Omega_r(x)$ is defined on $0< x\le 1$ and $r\in\R$ by:

$\Omega_r(x)=1+x^{r}$ if $r\ne 0$, and $\Omega_r(x)=\log(2/x)$ if $r=0$.
\end{thm}

Since this theorem is a technical tool, we had rather postponed its proof to the end of the paper.

\begin{prop}\label{prop:PNMmaps}
Let $1< p,q<\infty$, $0\le s\le n/p$  
and $s'<  s$, satisfying either 
$-N<s'$ or $N=s'=0$ and $q\le 2$.

\begin{enumerate}
	\item \label{item:PNMmaps1}  If $s>s'+N-M$, then $\PN^{N,M}$ is a bounded linear operator  from 
$Mult(F^{p,q}_s\to F^{p,q}_{s'})$ to $Mult(F^{p,q}_s\to F^{p,q}_{s'+N-M})$.
 
	\item \label{item:PNMmaps2} If $s=s'+N-M$, then for any $s''<s$ the operator   
$\PN^{N,M}$ is bounded from $Mult(F^{p,q}_s\to F^{p,q}_{s'})$ to $Mult(F^{p,q}_s\to F^{p,q}_{s''})$.

\item \label{item:PNMmaps3} If $s<s'+N-M$, then  $\PN^{N,M}$ is bounded from 
$Mult(F^{p,q}_s\to F^{p,q}_{s'})$ to  $Mult(F^{p,q}_s)$.
\end{enumerate}
\end{prop}

\begin{proof} 
In order to prove the three assertions, we will  apply Theorem \ref{thm:mf} with $ t=s-s'$, $k>\max\{s,M+n-1\}$, 
$J>\max\{k+n+N,2k+s-s'\}$, $L>N+J$ and $m>\max\{s+N-M,N-t+n+J\}$, $f\in F^{p,q}_s$ and $g\in Mult(F^{p,q}_s\to F^{p,q}_{s'})$.

Since $0<t<n+N$, it is clear that $t,k, J, L$ and $M$ satisfy the conditions of that theorem. 

If $-N<s'$, then by Propositions \ref{prop:grown} \eqref{item:grown1} and \ref{prop:embedF} \eqref{item:embedF5}, 
 we have  $g\in F^{p,q}_{s'}\cap B^\infty_{s'-s}\subset B^1_{-N}\cap B^\infty_{s'-s}$.

If $s'= N=0$ and $1<q\le 2$, then $g\in F^{p,q}_{0}\cap B^\infty_{-s}\subset H^p\cap B^\infty_{-s}$.

Since $s\ge 0>k+n+N-J>k+t-N-J$, we have $f\in F^{p,q}_s\subset B^1_{k+t-N-J}$. Then, $f$ and $g$ are in the conditions of Theorem \ref{thm:mf}.

Now, we estimate the terms appearing in \eqref{eqn:mf}.

 We will first check that in any of the hypotheses \eqref{item:PNMmaps1}, \eqref{item:PNMmaps2} or \eqref{item:PNMmaps3},  the first group of terms $f\PN^{N+J,M+i}(g)$, with $0\le i\le k+1-n-M$, are in $F^{p,q}_s$, and consequently in any of the Triebel-Lizorkin spaces considered in the three statements.
  
 Indeed, since $J>2k+s-s'$, then  $N+J-M-i>k+s-s'$. Thus, Proposition \ref{prop:PNMonB} 
 and the fact that $g\in B^\infty_{s'-s}$ give $\PN^{N+J,M+i}(g)\in B^\infty_k$. 
 Hence, since $k>s$,  by Proposition \ref{prop:grown}\eqref{item:grown3},  
$\PN^{N+J,M+i}(g)\in  Mult(F^{p,q}_s)$ and  $f\PN^{N+J,M+i}(g)\in F^{p,q}_{s}$.

Consider now the second group the terms in \eqref{eqn:mf}, that is, the terms of the form $\PN^{N+J,M+i-j}\left(g\,d^j\,f(R,\overset{(j)}{\cdots},R)\right)$, 
with $ k+1-n-M<i\le J$. We will prove that in all the cases these terms are in $F^{p,q}_{s'+N-M}$.   
Observe that if $f\in F^{p,q}_s$ then $d^j\,f(R,\overset{(j)}{\cdots},R)\in F^{p,q}_{s-j}$. 
So,  by  Proposition \ref{prop:grown}\eqref{item:grown4},  
$gd^j\,f(R,\overset{(j)}\cdots,R)\in F^{p,q}_{s'-j}$ and,  since $s'-j>-N-J$, Proposition \ref{prop:PNMonF} gives 
$$
\PN^{N+J,M+i-j}\left(g\,d^j\,f(R,\overset{(j)}{\cdots},R)\right)\in F^{p,q}_{s'+N-M+J-i}
\subset F^{p,q}_{s'+N-M}.
$$
And in the three cases, we obtain the desired conclusion, with the corresponding estimates of the norms.

To conclude, we consider the term $Q^{N,M,k}(f,g)$. Now, the three statements have to be treated separately.

Assume first, that we are in the conditions of statement \eqref{item:PNMmaps1},  $s>s'+N-M$. We want to prove that $ Q^{N,M,k}(f,g)\in F^{p,q}_{s'+N-M}$, that is  
\begin{equation*}
\Psi_1(z):=(1-|z|^2)^{m-s'-N+M}(1+R)^m  Q^{N,M,k}(f,g)(z)\in T^{p,q}.
\end{equation*}

By Proposition \ref{prop:grown}\eqref{item:grown1}, $g\in B^\infty_{s'-s}$, 
and if $l>s$ is a positive integer, then  
$\varphi(u):=|R^l_{n+L}f(u)|(1-|u|^2)^{l-s}\in T^{p,q}$. Thus, Theorem \ref{thm:mf} and 
Lemma \ref{lem:PPF} give 
\begin{align*}
&\|Q^{N,M,k}(f,g)\|_{F^{p,q}_{s'+N-M}}\approx \|\Psi_1\|_{T^{p,q}}\\
&\lesssim 
\|g\|_{B^\infty_{s'-s}}\left\|\int_\B |\varphi(u)| 
\frac{(1-|z|^2)^{m-s'-N+M}(1-|u|^2)^{N+J+s'-k-1}}{|1-z\overline u|^{n+M+J-k+m}} d\nu(u)\right\|_{T^{p,q}}\\
&\quad+\|g\|_{B^\infty_{s'-s}} 
\left\|\int_\B |\varphi(u)|\frac{(1-|z|^2)^{k-s+1}(1-|u|^2)^{L+s-1}}
{|1-z\overline u|^{n+L+k+1}}d\nu(u)\right\|_{T^{p,q}}\\
&\lesssim \|g\|_{B^\infty_{s'-s}}\|\varphi\|_{T^{p,q}}\approx \|g\|_{B^\infty_{s'-s}}\|f\|_{F^{p,q}_s}
\lesssim \|g\|_{Mult(F^{p,q}_s\to F^{p,q}_{s'})}\|f\|_{F^{p,q}_s}.
\end{align*}

Combining this result with the estimates of the other terms we ends 
the proof of \eqref{item:PNMmaps1}.

Let us prove \eqref{item:PNMmaps2}, that is, the case  $N-M=s-s'$. Analogously to the above case,  
for $s''<s$ we need to show that  
\begin{equation*}
\Psi_2(z):=(1-|z|^2)^{m-s''}(1+R)^m  Q^{N,M,k}(f,g)(z)\in T^{p,q}.
\end{equation*}

Since $\log\frac {2}{1-|z|^2} \lesssim (1-|z|^2)^{s''-s}$, we have
\begin{align*}
\|\Psi_2\|_{T^{p,q}}
&\lesssim 
\|g\|_{B^\infty_{s'-s}}\left\|\int_\B |\varphi(u)| 
\frac{(1-|z|^2)^{m-s}(1-|u|^2)^{N+J+s'-k-1}}{|1-z\overline u|^{n+M+J-k+m}} d\nu(u)\right\|_{T^{p,q}}\\
&\quad+\|g\|_{B^\infty_{s'-s}} 
\left\|\int_\B |\varphi(u)|\frac{(1-|z|^2)^{k-s+1}(1-|u|^2)^{L+s-1}}{|1-z\overline u|^{n+L+k+1}}d\nu(u)\right\|_{T^{p,q}},
\end{align*}
and by Lemma \ref{lem:PPF}, we obtain 
$$
\| Q^{N,M,k}(f,g)\|_{F^{p,q}_{s''}}\approx \|\Psi_2\|_{T^{p,q}}
\lesssim \|g\|_{B^\infty_{s'-s}}\|f\|_{F^{p,q}_s}
\lesssim \|g\|_{Mult(F^{p,q}_s\to F^{p,q}_{s'})}\|f\|_{F^{p,q}_s}.
$$ 

This concludes the proof of \eqref{item:PNMmaps2}.

If $N-M>s-s'$, then the same arguments used to prove the above cases give 
\begin{equation*}
\Psi_3(z):=(1-|z|^2)^{m-s}(1+R)^m  Q^{N,M,k}(f,g)(z)\in T^{p,q}.
\end{equation*}
Thus $Q^{N,M,k}(f,g)\in F^{p,q}_{s}$, which ends the proof. 
\end{proof}

\begin{cor}\label{cor:PNMMulttoCF}
Let $1< p,q<\infty$, $0\le s\le n/p$  
and $s'<  s$, satisfying either 
$-N<s'$ or $N=s'=0$ and $q\le 2$.

Let $k$ be a nonnegative integer satisfying $s>s'+N-M-k$, 
and $g\in Mult(F^{p,q}_s\to F^{p,q}_{s'})$.

Then, for any polynomial $p(x)$ of degree $k$, we have
$$
p(R)\,( \PN^{N,M}(g))\in Mult(F^{p,q}_s\to F^{p,q}_{s'+N-M-k}).
$$
If $s'=s-N+M$, then $\PN^{N,M}(g)\in CF^{p,q}_s$, that is, $\PN^{N,M}$ maps 
$Mult(F^{p,q}_s\to F^{p,q}_{s-N+M})$ to $CF^{p,q}_s$.
\end{cor}

\begin{proof}
Since 
$$
p(R)\,(\PN^{N,M}(g))=\sum_{j=0}^k a_j \PN^{N,M+j}(g),
$$
the first result follows from Proposition \ref{prop:PNMmaps}. 
Indeed, if $s'+N-M-j<s$, then we apply Proposition \ref{prop:PNMmaps}\eqref{item:PNMmaps1}, and if $s'+N-M-j\ge s$ then we apply Proposition \ref{prop:PNMmaps}\eqref{item:PNMmaps2},\eqref{item:PNMmaps3}.

The fact that if $s'=s-N+M$, then $\PN^{N,M}(g)\in CF^{p,q}_s$, follows from the above result applied to $p(x)=(1+x)^{k_s}$. 
\end{proof}

\begin{cor} \label{cor:TNMmaps}
Let $1< p,q<\infty$, $0\le s\le n/p$, $N\ge 0$, $M>-n$ and $s'<  s$, satisfying either 
$-N<s'$ or $N=s'=0$ and $q\le 2$.

Let $\T^{N,M}:F^{p,q}_{s'+N-M}\to F^{p,q}_{s'}$ be  the  inverse operator 
of $\PN^{N,M}:F^{p,q}_{s'}\to F^{p,q}_{s'+N-M}$ 

\begin{enumerate}
	\item \label{item:TNMmaps1} If $s>s'$ and $s>s'+N-M$, then $\T^{N,M}$ maps 
$Mult(F^{p,q}_s\to F^{p,q}_{s'+N-M})$ to $Mult(F^{p,q}_s\to F^{p,q}_{s'})$.

	\item \label{item:TNMmaps2} If $N-M<0$, then $\T^{N,M}$ maps $Mult(F^{p,q}_s\to F^{p,q}_{s+N-M})$ to $CF^{p,q}_s$. 
\end{enumerate}
\end{cor}

\begin{proof}
For any  integer $l>\max\{0,-M\}$, formula \eqref{eqn:expT} gives 
\begin{align*}
\T^{N,M}(g)=\sum_{j=0}^{l} a_j \PN^{M+l,N+j}\,(g).
\end{align*}
If $s>s'$ and $s>s'+N-M$, then the operators $\PN^{M+l,N+j}$ satisfies the 
properties of Proposition \ref{prop:PNMmaps}, and thus \eqref{item:TNMmaps1}  follows from this proposition.

If $N-M<0$, then \eqref{item:TNMmaps2} follows from Corollary \ref{cor:PNMMulttoCF}.
\end{proof}

\begin{rem}\label{rem:obsthm}
Observe that  Proposition \ref{prop:PNMmaps} and Corollary \ref{cor:TNMmaps}  
gives immediately the proof of Theorem \ref{thm:bijectMF}\eqref{item:bijectMF1}.
\end{rem}

\begin{prop}\label{prop:charMssp}
Let $1<p,q<\infty$,  $0\le s\le n/p$ and  $ s'\le s$. Then, the following assertions are equivalent:
\begin{enumerate}
	\item \label{item:charMssp1} $g\in Mult(F^{p,q}_s\to F^{p,q}_{s'})$ if $s'<s$, 
or $g\in CF^{p,q}_s$ if $s'=s$.
	\item \label{item:charMssp4} For some (any) $M_1,\cdots, M_k>0$, then 
	$$
	R^1_{M_k}\cdots R^1_{M_1}g\in  Mult(F^{p,q}_s\to F^{p,q}_{s-k}).
	$$
\end{enumerate} 

In particular \eqref{item:charMssp1} is equivalent to $(1+R)^kg\in Mult(F^{p,q}_s\to F^{p,q}_{s'-k})$ for some (any) positive integer $k$, and it is also equivalent to 
$R^k_M g\in Mult(F^{p,q}_s\to F^{p,q}_{s'-k})$ for some (any) positive integer $k$ 
and some (any) $M>0$.
\end{prop}

\begin{proof} We consider first the case $s'<s$. 
We want to show that $g\in Mult(F^{p,q}_s\to F^{p,q}_{ s'})$ 
if and only if $R^1_M g\in  Mult(F^{p,q}_s\to F^{p,q}_{ s'-1})$ for some (any) $M>0$.

Clearly, iterating this result we obtain the equivalence between \eqref{item:charMssp1} and \eqref{item:charMssp4}.

If $g\in Mult(F^{p,q}_s\to F^{p,q}_{ s'})$, then  
by Proposition \ref{prop:grown}\eqref{item:grown5} we have that  $R^1_M g\in  Mult(F^{p,q}_s\to F^{p,q}_{ s'-1})$ for any $M>0$.

Now assume that $R^1_M g\in  Mult(F^{p,q}_s\to F^{p,q}_{ s'-1})$ for some $M>0$. 
By Proposition \ref{prop:grown}\eqref{item:grown1}, 
$R^1_M g\in B^\infty_{ s'-s-1}$ and thus $g\in B^\infty_{ s'-s}$.    
Therefore, Corollary \ref{cor:RkvPN} with $N>\max\{0,-s\}$  gives,
$$
g=\sum_{j=1}^m a_j\PN^{N+j,N}(R^1_M g)+a_{m+1}\PN^{N+m,N}(g).
$$
 
Thus, if $s'-s+m>s$ Propositions \ref{prop:PNMmaps}, \ref{prop:PNMonB} and \ref{prop:grown}\eqref{item:grown3} give  
 $$
 g\in Mult(F^{p,q}_s\to F^{p,q}_{ s'})+B^\infty_{ s'-s+m}
 \subset Mult(F^{p,q}_s\to F^{p,q}_{ s'}),
 $$ 
which proves the above mentioned equivalence.

Now we use the above results to prove the case $s'=s$. 

If $g\in CF^{p,q}_s$, that is $(1+R)^{k_s}g\in Mult(F^{p,q}_s\to F^{p,q}_{s-k_s})$, 
then the above  equivalence for the case $s'<s$ with $s'=s-1$, gives that $g\in CF^{p,q}_s$ 
if and only if $(1+R)g\in Mult(F^{p,q}_s\to F^{p,q}_{s-1})$.
Now the above proved equivalence for the case $s'<s$  gives 
$(1+R)g\in Mult(F^{p,q}_s\to F^{p,q}_{s-1})$
if and only if
$$ 
R^1_{M_k}\cdots R^1_{M_1}(1+R) g=(1+R)R^1_{M_k}\cdots R^1_{M_1}g\in  Mult(F^{p,q}_s\to F^{p,q}_{s-k-1})
$$
 for some (any) $M_1,\cdots,M_k>0$. This is also equivalent to 
$$
R^1_{M_k}\cdots R^1_{M_1}g\in  Mult(F^{p,q}_s\to F^{p,q}_{s-k}),
$$
which concludes the proof.
\end{proof}

\begin{cor}
$(1+R)^k$ and $R^k_M$ are bijective operators from $Mult(F^{p,q}_s\to F^{p,q}_{s'})$ 
to $Mult(F^{p,q}_s\to F^{p,q}_{s'-k})$ if $s'<s$, and from $CF^{p,q}_s$ to $Mult(F^{p,q}_s\to F^{p,q}_{s-k})$.
\end{cor}

\subsection{Proof of Theorem \ref{thm:bijectMF}}
\begin{proof}

As we have observed in Remark \ref{rem:obsthm}, the proof of \ref{thm:bijectMF}\eqref{item:bijectMF1}
 follows from Proposition \ref{prop:PNMmaps} and Corollary \ref{cor:TNMmaps}.

To prove \eqref{item:bijectMF3}, we consider now $N-M<0$.
In Corollary \ref{cor:TNMmaps} we have proved that $\T^{N,M}$, the inverse operator  of $\PN^{N,M}$, maps $Mult(F_s^{p,q}\to F_{s+N-M}^{p,q})$ to $CF_s^{p,q}$. Hence, 
it is sufficient to prove that $\PN^{N,M}$ maps $CF^{p,q}_s$ to 
$Mult(F^{p,q}_s\to F^{p,q}_{s+N-M})$.
 
If $g\in CF^{p,q}_s$, by Proposition \ref{prop:charMssp}, $R^k_{n+N}g\in Mult(F^{p,q}_s\to F^{p,q}_{s-k})$.  Since $\PN^{N,M}(g)=\PN^{N+k,M}(R^k_{n+N}g)$, 
  by Proposition \ref{prop:PNMmaps} we obtain 
$\PN^{N,M}(g)\in Mult(F^{p,q}_s\to F^{p,q}_{s+N-M})$,
which concludes the proof of \eqref{item:bijectMF3}.

In order to prove \eqref{item:bijectMF2}, we consider now $N-M>0$.  By Corollary \ref{cor:PNMMulttoCF}, $\PN^{N,M}$ is a bounded operator from  $Mult(H^p_s\to F^{p,q}_{s-N+M})$ to $CF^{p,q}_s$. In the other direction, since formula \eqref{eqn:expT} gives $ \T^{N,M}=\sum_{j=0}^{l} a_j \PN^{M+l,N+j}$, Proposition \ref{prop:PNMmaps} gives that 
 $\T^{N,M}$ maps $CF^{p,q}_s$ to $Mult(F^p_s\to F^{p,q}_{s-N+M})$. 
\end{proof}

\subsection{Characterizations of $Mult(F^{p,q}_s)$}

The results in the above sections provide a 
description of the space of multipliers of $F_s^{p,q}$ that we summarize in the following theorem.

\begin{thm} \label{thm:charmult}
Let $1<p,q<\infty$ and $0\le s\le n/p$. Then the following assertions are equivalent.

\begin{enumerate}
	\item\label{item:charmult1} $g\in Mult(F^{p,q}_s)$. 
	\item\label{item:charmult2}  $g\in H^\infty\cap CF^{p,q}_s$.
	\item\label{item:charmult3}  $g\in H^\infty$ and for some (any) $M>N\ge 0$ satisfying either 
	$N>0$ or $N=s= 0$ and $q\le 2$, then $\PN^{N,M}(g)\in Mult(F^{p,q}_s\to F^{p,q}_{s+N-M})$.
	\item \label{item:charmult4} $g\in H^\infty$ and for some (any) $M_1,\cdots, M_k>0$, then 
	$R^1_{M_k}\cdots R^1_{M_1}g\in Mult(F^{p,q}_s\to F^{p,q}_{s-k})$. 
	\item \label{item:charmult5} $g\in H^\infty$ and for some (any) positive integer $k>s$, then the measure $d\mu_g(z)=|(1+R)^kg(z)|^p(1-|z|^2)^{(k-s)p-1}d\nu(z)$ is a Carleson measure in the sense that  $F_s^{p,q}$ is embedded in $T^{p,q}( \mu_g)$.
\end{enumerate}
\end{thm}

\begin{proof} The equivalence between \eqref{item:charmult1} and \eqref{item:charmult2} follows 
from Corollary \ref{cor:charMFs} and Proposition \ref{prop:charMssp}. 
The equivalence between \eqref{item:charmult2} and \eqref{item:charmult3} follows 
from  Theorem \ref{thm:bijectMF}, and the equivalence between 
\eqref{item:charmult2} and \eqref{item:charmult4} follows from Proposition \ref{prop:charMssp}.

The last statement is equivalent to $(1+R)^k g\in Mult(F^{p,q}_s\to F^{p,q}_{s-k})$, which corresponds to the case $M_j=1$ in \eqref{item:charmult4}.
\end{proof}

\section{Multipliers of Hardy-Sobolev spaces}\label{sec:multHS}
 \subsection{Proof of Theorem \ref{thm:bijectMH}} Let us begin with a lemma.

\begin{lem} \label{prop:CsM}
Let $1<p<\infty$ and $0<s\le n/p$. 
For $g\in H^1$, the following assertions are equivalent.

\begin{enumerate}
	\item $g\in Mult(H^p_s\to H^p)$, that is $g\in H^p$ and $|g|^p d\sigma$ is a Carleson measure for $H^p_s$ on $\bB$.
	\item $\mathcal{C}_s(g)\in CF^{p,2}_s$.
\end{enumerate}
\end{lem}

\begin{proof}
This is a simple reformulation of Theorem \ref{thm:bijectMF} \eqref{item:bijectMF2}, applied to $F_s^{p,2}=H_s^p$, $F_s^{p,0}=H^p$ and $\PN^{0,-s}={\mathcal C}_s$.
\end{proof}

 \begin{proof}[Proof of Theorem \ref{thm:bijectMH}] 

The equivalence \eqref{item:bijectMH1} $\Leftrightarrow$\eqref{item:bijectMH4} is just a reformulation of the equivalence \eqref{item:charmult1} $\Leftrightarrow$ \eqref{item:charmult3} in Theorem \ref{thm:charmult}, for $N=0$ and $q=2$.
For the implication \eqref{item:bijectMH2}$\Rightarrow$ \eqref{item:bijectMH1}, if $g={\mathcal C}_s(h)$, we apply Lemma \ref{prop:CsM} to the function $h$ and obtain that $g\in H^\infty\cap CF^{p,2}_s$. Theorem \ref{thm:charmult} gives \eqref{item:bijectMH1}.
Finally, let us prove that  \eqref{item:bijectMH1}$\Rightarrow$ \eqref{item:bijectMH2}. Again by Theorem \ref{thm:charmult}, we have that $g\in H^\infty\cap CF^{p,2}_s$, and Theorem \ref{thm:bijectMF} gives then that $g=\PN^{0,-s}(h)$, with $h\in Mult(H^p_s\to H^p)$.
\end{proof}

\subsection{The case $0<n-sp<1$}

Theorem \ref{thm:bijectMH},\eqref{item:bijectMH2}, gives a characterization of $g\in Mult(H_s^p)$ in terms on one hand of the boundedness of the function $g$, on the other hand of the existence of a function $h\in H^p$ such that $g={\mathcal C}_s(h)$ and $|h|^pd\sigma$ is a Carleson measure on $\bB$ for $H_s^p$.  
If $0<n-sp<1$ we can give a characterization of these measures in terms of capacities.

We recall that if $E\subset{{\bB}}$, $1<p<\infty$ and $0<s<n$, the nonisotropic Riesz capacity of the set $E$ is given by
$$C_{s,p}(E)= \inf \{||f||_{L^p}^p\,;\, f\geq0,\, \mathcal{I}_s(f)\geq 1 \,\,{\rm on}\,\, E\,\}.$$
 
 \begin{defn}
 If $1<p<\infty$ and $0<s<n$, we say that a finite positive Borel measure $\mu$ on ${\bB}$ is  a trace measure for the Hardy-Sobolev  space $H_s^p$ if  there exists $C>0$ such that for any $f\in H_s^p$,
\begin{equation}\label{qcarlesonmeasure}\|M_{rad}[f]\|_{L^p(d\mu)}\leq C\|f\|_{H_s^p},\end{equation}
where $M_{rad}[f](\zeta)=\sup_{r<1}|f(r\zeta)|$.
\end{defn}
 \begin{defn}
 If $1<p<\infty$ and $0<s<n$, we say that a finite positive Borel measure $\mu$ on ${\bB}$ is  a trace measure for the    space $\mathcal{I}_s[L^p]$ if  there exists $C>0$ such that for any $f\in L^p$,
\begin{equation}\label{qcarlesonmeasure2}\|\mathcal{I}_s[f]\|_{L^p(d\mu)}\leq C\|f\|_{L^p}.\end{equation}

\end{defn}
We have the following theorem
\begin{thm} \label{thm:n-sp<1}
Let $1<p<\infty$ and $0<n-sp<1$. If $g\in H^p$  and $k>s$, then we have that the following conditions are equivalent:

\begin{enumerate}
	\item The measure $d\mu=|g|^pd\sigma$ is a trace measure for $H_s^p$.
	\item The measure $d\mu=|g|^pd\sigma $ is a trace measure for $\mathcal{I}_s[L^p]$.
	\item\label{item:n-sp<1} The measure $d\mu=|g|^pd\sigma$ is a Carleson measure on $\bB$ for $H_s^p$.
	\item There exists $C>0$ such that for any closed set $E\subset \bB$, $\mu(E)\leq C C_{s,\,p}(E)$.
	\item If  $d\mu_g(z)=|(1+R)^k (\mathcal{C}_s(g)(z))|^2 (1-|z|^2)^{2(k-s)-1}d\nu(z)$, then   $H^p_s\subset  T^{p,2}(d\mu_g)$.
	\item If $d\mu_g(z)=|(1+R)^k (\mathcal{C}_s(g)(z))|^2 (1-|z|^2)^{2(k-s)-1}d\nu(z)$ then   $\mathcal{I}_s[L^p]\subset  T^{p,2}(d\mu_g)$.
\end{enumerate}

\end{thm}
\begin{proof}
The fact that (i) and  (iv) are equivalent, is proved in \cite{cohnverbitsky}. The equivalence of (ii) and  (iv) of the trace measures for $\mathcal{I}_s[L^p]$ can be deduced from the corresponding result for $\R^n$ (see for instance the book \cite{adamshedberg} and the references therein). Condition (i) implies (iii), and the proof of \cite{cohnverbitsky} can be easily adapted to show that (iii) implies (iv). The equivalence between (iii) and (v) is a consequence of Propositions \ref{prop:CsM} and \ref{prop:charMssp} and does not require the extra condition $n-sp<1$. Finally the equivalence of (v) and (vi) was proved in \cite{cascanteortega3}.
\end{proof}

Let us see that in the last theorem, we can give another description of the fact that $f\in Mult(H_s^p)$.
 We recall that a nonnegative measurable weight  $w$ on $\bB$ is in $A_p$, $1<p<\infty$,  if there exists $C>0$ such that for any nonisotropic ball $B\subset \bB$, $B=B(\zeta,r)=\{\eta\in{\bB};\, |1-\zeta\overline{\eta}|<2r\}$, $\zeta\in\bB$, $r<1$, 
$$\left( \frac{1}{|B|}\int_B w d\sigma\right)\left( \frac{1}{|B|}\int_B w^{\frac{-1}{p-1}}d\sigma \right)^{p-1}\leq C.$$

A weight $w$ is in $A_1$ if ${\displaystyle Mw(\zeta):=\sup_{B\ni \zeta} \frac{\int_Bw}{|B|}\lesssim w(\zeta)}$, a.e. $\zeta\in\bB$.

The following result can be found in \cite{mazyaverbitsky}.
\begin{prop}\label{prop:tracemeasure3}
Let $g$ be an integrable function on $\bB$ such that $|g|^pd\sigma$ is a trace measure for $\mathcal{I}_s[L^p]$.
Let $h$ be a measurable function on $\bB$ satisfying that there exists $C>0$ such that for any  weight $w$ in $A_1$,
\begin{equation}\label{eqn:weightestimate}
\int_{\bB} |h|^pw\leq C \int_{\bB} |g|^p w.
 \end{equation}
We then have that the measure $|h|^pd\sigma$ is a trace measure for $\mathcal{I}_s[L^p]$.
\end{prop}
As a  consequence of the above proposition and the fact that if $g\in L^p$ and $|g|^pd\sigma$ is a trace measure for ${\mathcal I}_s(L^p)$, we also have that the measure $|{\mathcal C}(g)|^pd\sigma$, where ${\mathcal C}$ is the Cauchy transform, is a trace measure for ${\mathcal I}_s(L^p)$, and, in particular, it is a Carleson measure for $H_s^p$, we deduce the following corollary.
\begin{cor} 
Let $1<p<\infty$ and $0<n-sp<1$. We then have that $f\in Mult(H_s^p)$ if and only if $f\in H^\infty$ and there exists $g\in L^p$ such that $f={\mathcal C}_s(g)$ and
the measure $d\mu=|g|^pd\sigma $ is a trace measure for $\mathcal{I}_s[L^p]$.
\end{cor}

\section{Multipliers and holomorphic potentials}\label{sec:multholpot}

Consider, in $\R^n$, the space of Bessel potentials of functions in $L^p(\R^n)$, $G_s[L^p(\R^n)]$, $1<p<\infty$ and $0<sp\leq n$. It is proved in \cite{boe}, that if $\mu$ is a positive measure on $\R^n$ and if $G_s$ is a Bessel potential, and   the nonlinear potential of $\nu$ defined by $V_{s,\,p}[\mu]= G_s*(G_s*\mu)^{p'-1}$ is bounded, then $V_{s,\,p}[\mu]$ is a multiplier on $G_s[L^p]$.

In this section we will study the analogous problem for nonisotropic holomorphic potentials. 
We recall some definitions and results.

If $\mu$ is a positive Borel measure on $\bB$,   $1<p<\infty$, $0<s<n$ and $w$ is an $A_p$-weight, the $(s,p)$-energy of $\mu$ with weight $w$ (see \cite{adams}), is defined by
 \begin{equation}\label{weightedenergy}
 {\mathcal E}_{s,p,w}(\mu)=\int_{\bB} (\mathcal{I}_s(\mu)(\zeta))^{p'}w(\zeta)^{-(p'-1)}d\sigma(\zeta).
 \end{equation} 

 It is also introduced in \cite{adams} a weighted Wolff-type potential of a measure $\mu$ as
 \begin{equation}\label{weightedwolffpotential}
 {\mathcal W}_{s,p,w}(\mu)(\zeta)= \int_0^1\left( \frac{\mu(B(\zeta,1-r))}{(1-r)^{n- sp}}\right)^{p'-1} {\int\!\!\!\!\!\setminus}_{B(\zeta, 1-r)}w^{-(p'-1)}(\eta)d\sigma(\eta)\frac{dr}{1-r}.
 \end{equation}
 Here
 $${\int\!\!\!\!\!\setminus}_{B } \theta:= \frac1{|B |}\int_{B}\theta.$$
 
 We have the pointwise estimate 
 \begin{equation}\label{eqn:realpot}{\mathcal W}_{s,p,w}(\mu)(\zeta)\leq C\mathcal{I}_s[(w^{-1}  \mathcal{I}_s(\mu))^{p'-1}](\zeta).\end{equation} The weighted version of Wolff's theorem gives that the converse is true, provided we integrate with respect to $\mu$, that is,
  if $w$ in an $A_p$-weight, the following weighted Wolff-type theorem holds:
 \begin{equation}\label{weightedwolfftheorem}
 {\mathcal E}_{s,p,w}(\mu)= \int_{\bB}\mathcal{I}_s[(w^{-1}  \mathcal{I}_s(\mu))^{p'-1}] (\zeta) d\mu(\zeta)\simeq \int_{\bB} {\mathcal W}_{s,p,w}(\mu)(\zeta) d\mu(\zeta).
 \end{equation} 
  When $w\equiv1$ we will just write ${\mathcal W}_{s,p}(\mu)$ and ${\mathcal E}_{s,p}(\mu)$.

We also recall an extremal theorem for the nonisotropic Riesz capacities. See, for instance, the books of \cite{adamshedberg} and \cite{mazyashaposnikova} and the paper \cite{hedbergwolff}.

\begin{prop}\label{prop:extremalcapacity} Let $1<p<\infty$, $0<s<n/p$ and $G\subset{\bB}$ be an open set. There exists a positive capacitary measure $\mu_G$ such that
 \begin{enumerate}
 \item[(i)] ${\rm supp}\,\, \mu_G\subset \overline{G}$.
 
 \item[(ii)] $\mu_G(\overline{G})=C_{s,p} (G)={\mathcal E}_{s,p} (\mu_G)$.
 
 \item[(iii)] ${\mathcal W}_{s,p} (\mu_G)(\zeta)\geq C$, for every $\zeta\in G$.
 
 \item[(iv)] $\mathcal{I}_s((\mathcal{I}_s (\mu_G))^{p'-1})(\zeta)\leq C$, for any $\zeta\in \bB$.
 \item[(v)] If $\epsilon= \min(1,p-1)$, $$C_{s,p}(\{\zeta\in\bB;\, \mathcal{I}_s(\mathcal{I}_s(\mu_G)^{p'-1})(\zeta)\geq t\}) \lesssim \frac{C_{s,p}(G)}{t^\epsilon}.$$
 \item[(vi)] If $1<p\leq 2-\frac{s}{n}$ and $1<\delta<\frac{n}{n-s}$ or $2-\frac{s}{n}<p$ and $1<\delta<\frac{(p-1)n}{n-sp}$, then the weight $w^\delta=(\mathcal{I}_s(\mathcal{I}_s(\mu_G)^{p'-1}))^\delta$ is in $A_1$. Moreover, for any $\eta \in \bB$ and $y>0$, \begin{equation}\label{eqn:a1}\frac{w^\delta(B(\eta,y))}{y^n}\lesssim w^\delta(\eta).\end{equation}
 \end{enumerate}
\end{prop}

The methods in \cite{cascanteortegaverbitsky} can be adapted to the nonisotropic case to show the following theorem.

\begin{prop}\label{prop:tracemeasure1}

 Let $1<p<\infty$, $0<s<n/p$ and let $\mu$ be a  finite positive measure in $\bB$. Then the measure $d\mu_1=\frac{d\mu}{{\mathcal W}_{s,p}(\mu)^{p-1}}$ is a trace measure for $\mathcal{I}_s[L^p]$.
 \end{prop}

We also recall the nonisotropic versions of Theorem  2.1  and Lemma 3.1 in \cite{mazyaverbitsky}.

\begin{prop}\label{prop:tracemeasure2}
Let $1<p<\infty$, $0<s<n/p$ and let $\mu$ be a positive finite Borel measure on $\bB$. Then the following assertions are equivalent:
\begin{enumerate}
\item The measure $\mu$ is a trace measure for $\mathcal{I}_s[L^p]$.
\item The measure $(\mathcal{I}_s(\mu))^{p'}d\sigma$ is a trace measure for $\mathcal{I}_s[L^p]$.
\end{enumerate} 
\end{prop}

\begin{thm}\label{thm:multex}
Let $1<p<\infty$, $0<s<n$ and $\mu$ a finite positive Borel measure on $\bB$. Assume that $\mathcal{I}_s(\mathcal{I}_s(\mu)^{p'-1})$ is bounded on $\bB$. Then ${\mathcal C}_s(\mathcal{I}_s(\mu)^{p'-1})$ is a multiplier for $H_s^p$.
\end{thm}
\begin{proof}
 Let us first observe that the boundedness of the nonisotropic potential $\mathcal{I}_s(\mathcal{I}_s(\mu)^{p'-1})$ gives that the nonlinear potential ${\mathcal W}_{s,p}(\mu)$ is also bounded. Hence $\mu$ is a trace measure for $\mathcal{I}_s[L^p]$.
Indeed
\begin{equation}\label{eqn:multex}\begin{split}&
\int_{\bB} \mathcal{I}_s(f)^p(\zeta) d\mu(\zeta)= \int_{\bB} \mathcal{I}_s(f)^p(\zeta){\mathcal W}_{s,p}(\mu)^{p-1}(\zeta)\frac{d\mu(\zeta)}{{\mathcal W}_{s,p}(\mu)^{p-1}(\zeta)}\\&\lesssim \int_{\bB} \mathcal{I}_s(f)^p(\zeta)\frac{d\mu(\zeta)}{{\mathcal W}_{s,p}(\mu)^{p-1}(\zeta)}\lesssim \|f\|_{L^p}^p,
\end{split}\end{equation}
where in the last estimate we have used that by Proposition \ref{prop:tracemeasure1}, the measure $\displaystyle{\frac{d\mu}{{\mathcal W}_{s,p}(\mu)^{p-1}}}$ is a trace measure for $\mathcal{I}_s[L^p]$. 
By Proposition \ref{prop:tracemeasure2}, we also have that the measure $((\mathcal{I}_s(\mu)^{p'-1})^pd\sigma$ is also a trace measure for $\mathcal{I}_s[L^p]$.

Next, we have that for any $w\in A_1$
$$\int_{\bB} |{\mathcal C}(\mathcal{I}_s(\mu)^{p'-1})|^pwd\sigma \lesssim \int_{\bB} |(\mathcal{I}_s(\mu)^{p'-1})|^pwd\sigma.$$ Consequently, Proposition \ref{prop:tracemeasure3} gives that
$|{\mathcal C}(\mathcal{I}_s(\mu)^{p'-1})|^pd\sigma$ is also a trace measure for $\mathcal{I}_s[L^p]$. Of course, from the hypothesis, ${\mathcal C}_s(\mathcal{I}_s(\mu)^{p'-1})$ is bounded. The Proposition \ref{prop:CsM} gives then that
the holomorphic function
$\mathcal{C}_s(\mathcal{I}_s(\mu)^{p'-1})= \mathcal{C}_s({\mathcal C}(\mathcal{I}_s(\mu)^{p'-1}))$ is a multiplier for $H_s^p$.

\end{proof}  
When $0<n-2s<1$, we deduce that
\begin{cor}\label{prop:multiplierscas2}
Let $\mu$ be a positive finite Borel measure on $\bB$ and assume that $0<n-2s<1$. If ${\mathcal C}_s({\mathcal C}_s(\mu))$ is bounded, then it is a multiplier for $H_s^2$.
\end{cor}
\begin{proof}

It is proved in \cite{ortegafabrega}, that there exist a constant $C>0$ and a bounded kernel $L(z,\zeta)$,  such that
$${\mathcal C}_s({\mathcal C}_s(\mu))=C {\mathcal C}_{2s}(\mu)+{\mathcal L}(\mu),$$
where ${\mathcal L}(\mu)(z)=\int_{\bB} {\mathcal L}(z,\zeta)d\mu(\zeta).$
Consequently, ${\mathcal C}_{2s}(\mu)$ is also bounded. Next, the fact that $n-2s<1$ gives that $|{\mathcal C}_{2s}(\mu)|\approx {\mathcal I}_{2s}(\mu)\approx {\mathcal I}_s({\mathcal I}_s(\mu))$, and in particular, ${\mathcal I}_{s}(\mathcal{I}_s(\mu))$ is also bounded. The proof of the above Theorem, gives that $|\mathcal{I}_s(\mu)|^2$ is a trace measure for $\mathcal{I}_s[L^2]$. Consequently, $|{\mathcal C}_s(\mu)|^2$ is also a trace measure for $\mathcal{I}_s[L^2]$, and again applying Proposition \ref{prop:CsM}, we deduce that ${\mathcal C}_s({\mathcal C}_s(\mu))$ is a multiplier for $H_s^2$.
\end{proof}

For the case $0<n-sp<1$, $1<p<\infty$, it have been introduced by \cite{cohnverbitsky}, in the context of the study of exceptional sets for $H_s^p$, two families of holomorphic potentials that allow to extend some of the results for $p=2$ with ${\mathcal C}_s({\mathcal C}_s(\mu))$ (one more suitable for $p>2$ and the other one for $p\leq 2$). 

 Let $\mu$ be a finite positive Borel measure on ${\bB}$. For any $n-sp<\lambda<1$,  we set the following holomorphic functions on $\B$ defined by
 \begin{equation}\begin{split}\label{holompot1}
 &{\mathcal U}_{s,p,\lambda}(\mu)(z)\\
 &=\int_0^1 \int_{\bB}\left( \frac{\mu(B(\zeta, 1-r))}{(1-r)^{n-sp}}\right)^{p'-1} \frac{(1-r)^{\lambda-n}}{ (1-rz\overline{\zeta})^\lambda} d\sigma(\zeta) \frac{dr}{1-r},
 \end{split}\end{equation}
 and
 \begin{equation}\begin{split}\label{holompot2}
 &{\mathcal V}_{s,p,\lambda}(\mu)(z)\\
 &=\int_0^1 \left(\int_{\bB} \frac{(1-r)^{\lambda+sp-n}}{(1-rz\overline{\zeta})^\lambda}  d\mu(\zeta) \right)^{p'-1}  \frac{dr}{1-r}.
 \end{split}\end{equation}

It is   proved in \cite{cohnverbitsky} the following proposition:

 \begin{prop}\label{prop:pgreaterthan2}
  Let $1< p<\infty$, $0<s$ and $\lambda>0$ such that $0<n-sp<\lambda <1$. We then have:
  \begin{enumerate}
   \item If $1<p<2$   there exist $C_1,C_2>0$   such that for any    finite positive Borel measure $\mu$ on $\bB$ the following assertions hold:
  \begin{enumerate}
  \item[(a)] For any $\eta\in \bB$,
  $$\liminf_{\rho\rightarrow 1} {\rm Re}\,\, {\mathcal U}_{s,p,\lambda}(\mu)(\rho\eta)\geq C_1 {\mathcal W}_{s,p}(\mu)(\eta).$$
  \item[(b)] $|| {\mathcal U}_{s,p,\lambda}(\mu)||_{H_s^p }^p \leq C_2{\mathcal E}_{s,p} (\mu)$.
  \end{enumerate}
 
  \item If $p\geq 2$,  there exist $C_1,C_2>0$ such that for any  finite positive Borel measure $\mu$ on $ \bB$ the following assertions hold: 
  \begin{enumerate}
  \item[a)] For any $\eta\in \bB$,
  $$\liminf_{\rho\rightarrow 1} {\rm Re}\,\, {\mathcal V}_{s,p,\lambda}(\mu)(\rho\eta)\geq C_1 {\mathcal W}_{s,p} (\mu)(\eta).$$
  
  \item[b)] $|| {\mathcal V}_{s,p,\lambda}(\mu)||_{H_s^p }^p \leq C_2{\mathcal E}_{s,p} (\mu)$.
  \end{enumerate}
  \end{enumerate}
  \end{prop} 
\quad

\subsection{Proof of Theorem \ref{thm:multpot}}
\begin{proof}
We first observe that by Proposition \ref{prop:pgreaterthan2}, the hypothesis give that the nonisotropic Wolff potential ${\mathcal W}_{s,p}(\mu)$ is bounded on $\bB$.

Arguing as in \eqref{eqn:multex} in the proof of Theorem \ref{thm:multex},  this boundedness together with Proposition \ref{prop:tracemeasure1} gives that $\mu$ is a trace measure for $\mathcal{I}_s[L^p]$. Next, Proposition \ref{prop:tracemeasure2} gives that the measure $(\mathcal{I}_s(\mu))^{p'}d\sigma$ is a trace measure for $\mathcal{I}_s[L^p]$.

 {\bf Case $p< 2$.}
 We want to check that for any $f\in H_s^p$,
$$\|{\mathcal U}_{s,p,\lambda}(\mu)f\|_ {H_s^p} \lesssim \|f\|_{H_s^p}.$$

Since $F_s^{p,1}\subset H_s^p$, (see Proposition \ref{prop:embedF}),  it is enough to show that if $k>s$,
\begin{equation}\label{eqn:fp1}
\left\|  \int_0^1|(1+R)^k {\mathcal U}_{s,p,\lambda}(\mu)(\rho\zeta)||f(\rho\zeta)|(1-\rho)^{k-s} \frac{d\rho}{1-\rho} \right\|_{L^p(\bB)}\lesssim \|f\|_{H_s^p}.
\end{equation}
But 
\begin{equation*}\begin{split}&
\left\| \int_0^1 |(1+R)^k {\mathcal U}_{s,p,\lambda}(\mu)(\rho\zeta)|\,|f(\rho\zeta)|(1-\rho)^{k-s} \frac{d\rho}{1-\rho} \right\|_{L^p(\bB)}^p\\&
\leq  \int_{\bB} (M_{rad}(f)(\zeta))^p \left(\int_0^1|(1+R)^k {\mathcal U}_{s,p,\lambda}(\mu)(\rho\zeta)|(1-\rho)^{k-s} \frac{d\rho}{1-\rho}\right)^pd\sigma(\zeta)   .
\end{split}\end{equation*}
Thus \eqref{eqn:fp1} will follow if we show that the measure 
\begin{equation}\label{eqn:tracemeasurefp1}\left(\int_0^1|(1+R)^k {\mathcal U}_{s,p,\lambda}(\mu)(\rho\zeta)|(1-\rho)^{k-s} \frac{d\rho}{1-\rho}\right)^pd\sigma(\zeta)
\end{equation}  
is a trace measure for $H_s^p$.

Since we have that the measure $(\mathcal{I}_s(\mu))^{p'}d\sigma$ is a trace measure for $\mathcal{I}_s[L^p]$, Proposition \ref{prop:tracemeasure3} gives that it is enough that we show that for any  weight $w$ in $A_1$,
\begin{equation}\begin{split}\label{eqn:estimateweight}
&\int_{\bB}\left(\int_0^1|(1+R)^k {\mathcal U}_{s,p,\lambda}(\mu)(\rho\zeta)|(1-\rho)^{k-s} \frac{d\rho}{1-\rho}\right)^pw(\zeta)d\sigma(\zeta) \\&\lesssim \int_{\bB}(\mathcal{I}_s(\mu)(\zeta))^{p'}w(\zeta)d\sigma(\zeta).
\end{split}\end{equation}

 We have
\begin{equation*}\begin{split}&
\int_0^1 (1-\rho)^{k-s} |(1+R)^k {\mathcal U}_{s,p,\lambda}(\rho\eta)|\frac{d\rho}{1-\rho} 
\\&\lesssim 
\int_0^1 (1-\rho)^{k-s}\int_0^1 \int_{\bB} \left(\frac{\mu(B(\zeta, 1-r))}{(1-r)^{n-sp}} \right)^{p'-1} \frac{(1-r)^{\lambda-n}}{|1-r\rho\eta\overline{\zeta}|^{\lambda+k}} \frac{d\sigma(\zeta)drd\rho}{(1-r)(1-\rho)}.
\end{split}\end{equation*}
But
                              $$\int_0^1\frac{(1-\rho)^{k-s}}{|1-r\rho\eta\overline{\zeta}|^{\lambda+k}}\frac{d\rho}{1-\rho} \lesssim \frac1{|1-r\eta\overline{\zeta}|^{\lambda+s}}.$$

The preceding estimate, together with Fubini's theorem give that
$$\int_0^1 (1-\rho)^{(k-s)} |(1+R)^k {\mathcal U}_{s,p,\lambda}(\rho\eta)|\frac{d\rho}{1-\rho}\lesssim \Phi(\eta),$$
where 
$$\Phi(\eta)= \int_0^1\int_{\bB} \left(\frac{\mu(B(\zeta, 1-r))}{(1-r)^{n-sp}} \right)^{p'-1} \frac{(1-r)^{\lambda-n}}{|1-r\eta\overline{\zeta}|^{\lambda+s}} d\sigma(\zeta)\frac{dr}{1-r}.$$
We will follow some of the arguments in \cite{cohnverbitsky}, page 87. The key point is to apply Fubini's theorem in such a way that we obtain on one hand $\mu(B(\zeta,1-r))$ raised to the power $1$ and on the other hand an expression where in the denominator we have $|1-r\eta\overline{\zeta}|$ raised to some power strictly greater that $n$.
Precisely, let $\varepsilon>0$ such that $n-s<\varepsilon <\frac{\lambda+s-n(2-p)}{p-1}.$
We then have that
\begin{equation*}\begin{split}\Phi(\eta)&= \int_0^1\int_{\bB} \left(\frac{\mu(B(\zeta, 1-r))}{(1-r)^{n-s}}\frac{(1-r)^{\varepsilon-n}}{|1-r\zeta\overline{\eta}|^\varepsilon} \right)^{p-1}\\&
\times
\left(\frac{\mu(B(\zeta, 1-r))}{(1-r)^{n-s}}\right)^{p'-p}\frac{(1-r)^{\lambda+s-n(2-p)-\varepsilon(p-1)}}{|1-r\eta\overline{\zeta}|^{\lambda+s-\varepsilon(p-1)}}d\sigma(\zeta)\frac{dr}{1-r}.
\end{split}\end{equation*}
H\"older's inequality with exponents $1/(p-1)>1$ and $1/(2-p)$ gives
$$\Phi(\eta)\lesssim \Phi_1(\eta)^{p-1} \Phi_2(\eta)^{2-p},$$
where 
$$\Phi_1(\eta)=\int_0^1\int_{\bB}\frac{\mu(B(\zeta, 1-r))}{(1-r)^{n-s}}\frac{(1-r)^{\varepsilon-n}}{|1-r\zeta\overline{\eta}|^\varepsilon}d\sigma(\zeta)\frac{dr}{1-r}$$
and
\begin{equation*}\begin{split}&
\Phi_2(\eta)\\&=\int_0^1\int_{\bB}\left(\frac{\mu(B(\zeta, 1-r))}{(1-r)^{n-s}}\right)^{p'}
\frac{(1-r)^{(\lambda+s-n(2-p)-\varepsilon(p-1))/(2-p)}}{|1-r\eta\overline{\zeta}|^{(\lambda+s-\varepsilon(p-1))/(2-p)}}d\sigma(\zeta)\frac{dr}{1-r}.\end{split}\end{equation*}
Fubini's Theorem gives that 
$$\Phi_1(\eta)\lesssim \mathcal{I}_s(\mu)(\eta).$$

Applying again H\"older's inequality with exponents $\gamma=\frac1{(p-1)^2}>1$ and $\gamma'=\frac1{p(2-p)}$, we obtain that
$$
\int_{\bB}\Phi(\eta)^p w(\eta)d\sigma(\eta)
 \lesssim \left(\int_{\bB} \mathcal{I}_s(\mu)^{p'}(\eta) w(\eta)d\sigma(\eta)\right)^\frac1{\gamma} 
 \left(\int_{\bB}\Phi_2(\eta)w(\eta)d\sigma(\eta) \right)^\frac1{\gamma'}.
$$

Next,
\begin{equation*} 
\int_{\bB} \Phi_2(\eta)w(\eta)d\sigma(\eta) 
 \lesssim\int_0^1\int_{\bB} \left(\frac{\mu(B(\zeta, 1-r))}{(1-r)^{n-s}}\right)^{p'}\Omega(\zeta,r)d\sigma(\zeta)\frac{dr}{1-r},
\end{equation*}
where $$
\Omega(\zeta,r)=\int_{\bB}
\frac{(1-r)^{(\lambda+s-n(2-p)-\varepsilon(p-1))/(2-p)}}{|1-r\eta\overline{\zeta}|^{(\lambda+s-\varepsilon(p-1))/
(2-p)}}w(\eta)d\sigma(\eta)
$$
We choose $1<q$ such that   $nq<\frac{\lambda+s-\varepsilon(p-1)}{2-p}$ (which is possible since $n<\frac{\lambda+s-\varepsilon(p-1)}{2-p}$). We recall that any $w\in A_1$ satisfies a doubling condition of order $\tau$, for any $\tau>n$, that is, $w(2^kB)\lesssim 2^{k\tau} w(B)$. Choosing $\tau<\frac{\lambda+s-\varepsilon(p-1)}{2-p}$, by decomposing in coronas in the usual way the integral over $\bB$ with respect to the variable $\eta$, we obtain:
$$
\Omega(\zeta,r)\lesssim \frac{w(B(\zeta,1-r))}{(1-r)^n}.
$$
Altogether, since $w\in A_1$, we have that 
\begin{equation*}\begin{split}&
\int_{\bB} \Phi_2(\eta)w(\eta)d\sigma(\eta)\\&\lesssim\int_0^1 \int_{\bB} \left( \frac{\mu(B(\zeta,1-r))}{(1-r)^{n-s}}\right)^{p'}\frac{w(B(\zeta,1-r))}{(1-r)^n} d\sigma(\zeta)\frac{dr}{1-r}\\&
\lesssim \int_0^1 \int_{\bB} \left( \frac{\mu(B(\zeta,1-r))}{(1-r)^{n-s}}\right)^{p'}w(\zeta) d\sigma(\zeta) \frac{dr}{1-r}\\&
=\int_0^1 \int_{\bB} \left( \frac{\mu(B(\zeta,1-r))}{(1-r)^{n-s}}\right)^{p'-1}\frac{\int_{B(\zeta,1-r)}d\mu(\eta)}{(1-r)^{n-s}}w(\zeta) d\sigma(\zeta) \frac{dr}{1-r}.
\end{split}
\end{equation*}
Next, if $\eta\in B(\zeta,1-r)$, we have that $B(\zeta,1-r)\subset B(\eta, 4(1-r))$. This fact together with Fubini's Theorem  gives  that the above integral is bounded by
\begin{equation*}\begin{split}&
\int_0^1 \int_{\bB} \left( \frac{\mu(B(\eta,4(1-r)))}{(1-r)^{n-s}}\right)^{p'-1}\frac1{(1-r)^{n-s}}w(B(\eta,4(1-r))) d\mu(\eta) \frac{dr}{1-r}\\&
=\int_0^1 \int_{\bB} \left( \frac{\mu(B(\eta,4(1-r)))}{(1-r)^{n-sp}}\right)^{p'-1}\frac{w(B(\eta,4(1-r)))}{(1-r)^{n}} d\mu(\eta) \frac{dr}{1-r}\\&\approx\int_{\bB}{\mathcal W}_{s,p,w^{-(p-1)}}(\mu)d\mu\approx \int_{\bB} \mathcal{I}_s(\mu)^{p'}(\eta) w(\eta)d\sigma(\eta),
\end{split}
\end{equation*}
where in the last estimate we have used the weighted Wolff's theorem (see \cite{adams}). 
Altogether, we obtain that
$$
\int_{\bB}\Phi(\eta)^p w(\eta)d\sigma(\eta)
 \lesssim  \int_{\bB} \mathcal{I}_s(\mu)^{p'}(\eta) w(\eta)d\sigma(\eta), 
 $$
that is we have proved \eqref{eqn:estimateweight} and then the case $p<2$.

{\bf Case $p\geq 2$.} As in the preceding case, it is enough to prove that
$$\left( \int_0^1|(1+R)^k {\mathcal V}_{s,p,\lambda}(\mu)(\rho\zeta)|(1-\rho)^{k-s} \frac{d\rho}{1-\rho}\right)^pd\sigma$$ is a trace measure for $H_s^p$.

It is proved in \cite{cohnverbitsky}, page 90, that
\begin{equation}\label{est:cohnve}\begin{split}&\int_0^1|(1+R)^k {\mathcal V}_{s,p,\lambda}(\mu)(\rho\zeta)|(1-\rho)^{k-s} \frac{d\rho}{1-\rho}\\&\lesssim
\int_0^1 \left(    \frac{\mu(B(\zeta,\delta))}{\delta^{n-s}}\right)^{p'-1}\frac{d\delta}{\delta}
:=\Phi(\zeta).\end{split}\end{equation}

We will partially follow some of the arguments used in that paper to prove an extension of Wolff's inequality. 
We first observe that
$$\Phi^p(\zeta)= p \int_0^1 \left(\frac{\mu(B(\zeta,t))}{t^{n-s}}\right)^{p'-1}\left( \int_0^t \left( \frac{\mu(B(\zeta,y))}{y^{n-s}}\right)^{p'-1} \frac{dy}{y}\right)^{p-1}\frac{dt}{t}.$$
If we choose $0<\varepsilon<s$, H\"older's inequality with exponent $p-1>1$, gives that
$$
\left( \int_0^t \left( \frac{\mu(B(\zeta,y))}{y^{n-s}}\right)^{p'-1} \frac{dy}{y}\right)^{p-1}\lesssim t^\varepsilon 
\int_0^t   \frac{\mu(B(\zeta,y))}{y^{n-s}}  \frac{dy}{y^{1+\varepsilon}}.
$$
Thus we have that
\begin{equation*}
\Phi^p(\zeta) 
\lesssim 
\int_0^1 \int_0^t \mu(B(\zeta,y))\mu(B(\zeta,t))^{p'-1} \frac{dy}{y^{(n-s)+\varepsilon+1}}\frac{dt}{t^{(n-s)(p'-1)-\varepsilon+1}}.
\end{equation*}
 
Now we check that $\Phi^pd\sigma$ is a trace measure for $\mathcal{I}_s[L^p]$, proving that it satisfies that for any open set $G\subset \bB$, $\int_G \Phi^pd\sigma\lesssim C_{s,p}(G)$. 

Let  $\mu_G$ be the $(s,p)$-extremal capacitary measure of $G$, and $w=\mathcal{I}_s(\mathcal{I}_s(\mu_G)^{p'-1})$. Assertion (vi) in Proposition  \ref{prop:extremalcapacity} gives that since $p\geq 2$, for any $1<\delta<\frac{(p-1)n}{n-sp}$, $w^\delta\in A_1$, and for any $\eta \in \bB$ and $y>0$, $$\frac{w^\delta(B(\eta,y))}{y^n}\lesssim w^\delta(\eta).$$

 We fix $\delta$. Since by (iii) in Proposition \ref{prop:extremalcapacity}, $w^\delta \gtrsim 1$ on $G$, we have that
 $$
 \int_G \Phi^pd\sigma\lesssim \int_G \Phi^pw^\delta d\sigma.
 $$

Let $\eta\in B(\zeta,t)$. Then $B(\zeta,t)\subset B(\eta, 4t)$, and since $0<y<t$, we have
$$\mu(B(\zeta,y))\mu(B(\zeta,t))^{p'-1}\leq  \int_{B(\zeta,y)}\mu(B(\eta,4t))^{p'-1}d\mu(\eta).$$
Integrating with respect to $w^\delta(\zeta)d\sigma(\zeta)$ we have
\begin{equation*}\begin{split}&
\int_{\bB}\mu(B(\zeta,y))\mu(B(\zeta,t))^{p'-1}w^\delta(\zeta)d\sigma(\zeta)\\& 
\lesssim \int_{\bB}\int_{B(\zeta,y)}\mu(B(\eta,4t))^{p'-1}d\mu(\eta)w^\delta(\zeta)d\sigma(\zeta)\\
&\leq\int_{\bB} \mu(B(\eta,4t))^{p'-1}w^\delta(B(\eta,y))d\mu(\eta).
\end{split}\end{equation*}

Using the above estimates, we deduce that

\begin{equation*}\begin{split}&
\int_{G} \Phi(\zeta)^p d\sigma(\zeta)\\&
\lesssim
\int_{\bB}\int_0^1 \left(\frac{\mu(B(\eta,4t))}{t^{n-s}} \right)^{p'-1}\int_0^t \frac{w^\delta(B(\eta,y))dy}{y^{n-s+\varepsilon+1}}\frac{dt}{t^{1-\varepsilon}}d\mu(\eta).
\end{split}\end{equation*}

The estimate \eqref{eqn:a1} gives then that the above integral can be bounded, up to a constant by

\begin{equation*}\begin{split}&
\int_{\bB}\int_0^1 \left(\frac{\mu(B(\eta,4t))}{t^{n-s}} \right)^{p'-1}\int_0^t \frac{w^\delta( \eta) dy}{y^{ -s+\varepsilon+1}}\frac{dt}{t^{1-\varepsilon}}d\mu(\eta)\\&=
\int_{\bB}\int_0^1 \left(\frac{\mu(B(\eta,4t))}{t^{n-s}} \right)^{p'-1} w^\delta( \eta ) t^{s-\varepsilon}\frac{dt}{t^{1-\varepsilon}}d\mu(\eta)\\&=
\int_{\bB}\int_0^1 \left(\frac{\mu(B(\eta,4t))}{t^{n-sp}} \right)^{p'-1} w^\delta( \eta )  \frac{dt}{t}d\mu(\eta)\approx \int_{\bB} {\mathcal W}_{s,p}(\mu) (\eta)w^\delta(\eta) d\mu(\eta).
\end{split}\end{equation*}
Now, since ${\mathcal W}_{s,p}(\mu) $ is bounded on ${\bB}$,  then $\mu$ is a trace measure for ${\mathcal I}_s[L^p]$ (see the proof of Theorem \ref{thm:multex}) and we can bound the above integral, up to a constant, by
\begin{equation*}\begin{split}&
\int_{\bB}  w^\delta(\eta) d\mu(\eta)\\&
=\int_0^M \mu(\{\eta;\, w(\eta)\geq t\}) t^{\delta-1}dt\lesssim  \int_0^M C_{s,p}(\{\eta;\, w(\eta)\geq t\}) t^{\delta-1}dt.
\end{split}\end{equation*}

Since
$$ C_{s,p}(\{\eta;\, w(\eta)\geq t\})\lesssim \frac{C_{s,p}(G)}{t},$$ 
by (v) in  Proposition  \ref{prop:extremalcapacity}, this last integral is bounded, up to a constant,
by $$C_{s,p}(G)\int_0^M t^{\delta-2}dt\approx C_{s,p}(G).$$
That finishes the proof of the case $p\geq 2$.
\end{proof}

\begin{rem}\label{rem:c2s}
If $p=2$, we can replace the potentials ${\mathcal V}_{s,p,\lambda}(\mu)$ considered in the above theorem, by ${\mathcal C}_{2s}(\mu)$.  This observation is a consequence of the fact  that the estimate \eqref{est:cohnve} holds for ${\mathcal C}_{2s}(\mu)$, that is,
$$\int_0^1 |(1+R)^k {\mathcal C}_{2s}(\mu)(\rho\zeta)|(1-\rho)^{k-s}\frac{d\rho}{1-\rho}\lesssim \int_0^1 \frac{\mu(B(\zeta,\delta))}{\delta^{n-s}}\frac{d\delta}{\delta}.$$
The same arguments used to finish the case $p\geq 2$ for the potentials ${\mathcal V}_{s,p,\lambda}(\mu)$, prove that if ${\mathcal C}_{2s}(\mu)$ is bounded, then it is a multiplier for $H_s^p$. 
\end{rem}

\subsection{Proof of Theorem \ref{prop:multiplierscasp=2}}
\begin{proof}
Assertion \eqref{item:multiplierscasp=21} in Theorem \ref{prop:multiplierscasp=2} follows from Theorem \ref{thm:multex}. Assertion \eqref{item:multiplierscasp=22} is a consequence of Corollary \ref{prop:multiplierscas2} and Remark \ref{rem:c2s}.
\end{proof}

\section{Proof of Theorem \ref{thm:multsextremal}}\label{sec:examples}
\begin{proof}

We begin with the case $p<2$ and $n-s<1$. We will show that 
\begin{equation}\label{eq:pointcappotential}M_{rad}[{\mathcal U}_{s,p,\lambda}(\mu_E)]\lesssim 
\mathcal{I}_s[\mathcal{I}_s[\mu_E]^{p'-1}].\end{equation}
Since $\mu_E$ is the capacitary extremal measure associated to $E$, we have that $\mathcal{I}_s[\mathcal{I}_s(\mu_E)^{p'-1}]$ is bounded, and consequently, if \eqref{eq:pointcappotential} holds, ${\mathcal U}_{s,p,\lambda}(\mu_E)$ is also a bounded function. Theorem \ref{thm:multpot} gives then that the potential  ${\mathcal U}_{s,p,\lambda}(\mu_E)$ is a multiplier for $H_s^p$.
So we are led to show \eqref{eq:pointcappotential}.

Since $\lambda>n-s$, we have
\begin{equation*}\begin{split}&
|{\mathcal U}_{s,p,\lambda}(\mu_E)(z)| \\&
\approx \int_0^1\int_{\bB} \left(\frac{\mu_E(B(\zeta,1-r))}{(1-r)^{n-sp}} \right)^{p'-1}\frac{(1-r)^{\lambda-n}}{|1-rz\overline{\zeta}|^\lambda}d\sigma(\zeta)\frac{dr}{1-r}\\&
\leq \int_0^1\int_{\bB} \left(\frac{\mu_E(B(\zeta,1-r))}{(1-r)^{n-s}} \right)^{p'-1}\frac{1}{|1-z\overline{\zeta}|^{n-s}}d\sigma(\zeta)\frac{dr}{1-r}
\end{split}
\end{equation*}
Next,  since $\frac1{p'-1}<1$, and $\mu_E(B(\zeta,1-r))$ is a decreasing function on $r$, we have
$$
\int_0^1 \left(\frac{\mu_E(B(\zeta,1-r))}{(1-r)^{n-s}} \right)^{p'-1}\frac{dr}{1-r}\lesssim \left(\int_0^1\frac{\mu_E(B(\zeta, (1-r)))}{(1-r)^{n-s}}\frac{dr}{1-r}\right)^{p'-1}.$$
Hence 
\begin{equation*}\begin{split}&
|{\mathcal U}_{s,p,\lambda}(\mu_E)(z)| \\&
\lesssim \int_{\bB}\left(\int_0^1\frac{\mu_E(B(\zeta, (1-r)))}{(1-r)^{n-s}}\frac{dr}{1-r}\right)^{p'-1}\frac{d\sigma(\zeta)}{|1-z\overline{\zeta}|^{n-s}}\\&\approx \mathcal{I}_s[\mathcal{I}_s(\mu_E)^{p'-1}](z).
\end{split}\end{equation*}

In particular, we deduce  \eqref{eq:pointcappotential}.

Next we deal with the case $p\geq 2$. In this case we will check that
\begin{equation}\label{eq:pointcappotential2}M_{rad}[{\mathcal V}_{s,p,\lambda}(\mu_E)](\eta)\lesssim 
{\mathcal W}_{s,p}[\mu_E](\eta),\end{equation}
which again gives that since $\mu_E$ is the capacitary extremal function that ${\mathcal V}_{s,p,\lambda}(\mu_E)$ is a bounded and hence a multiplier for $H_s^p$.
We have that
\begin{equation*}\begin{split}&{\mathcal V}_{s,p,\lambda}[\mu_E](z)\\&\leq \int_0^1 \left(\int_{\bB} \frac{(1-r)^{\lambda+sp-n}}{|1-rz\overline{\zeta}|^\lambda}d\mu_E(\zeta)\right)^{p'-2}\int_{\bB} \frac{(1-r)^{\lambda+sp-n}}{|1-rz\overline{\xi}|^\lambda}d\mu_E(\xi)\frac{dr}{1-r}.\end{split}\end{equation*}

Now, write $z=\rho\eta$ and fix $\delta<1$. Since $|1-rz\overline{\zeta}|\approx (1-r) +(1-\rho) + |1-\eta\overline{\zeta}|$, we have
$$
\int_{\bB} \frac{d\mu_E(\zeta)}{|1-rz\overline{\zeta}|^\lambda}\geq \int_{B(\eta,\delta)} \frac{d\mu_E(\zeta)}{|1-rz\overline{\zeta}|^\lambda}\gtrsim \frac{\mu_E(B(\eta,\delta))}{(\delta+1-r+1-\rho)^\lambda}.$$
On the other hand,
\begin{equation*}\begin{split}&
\int_{\bB} \frac{d\mu_E(\xi)}{|1-rz\overline{\xi}|^\lambda}
\lesssim\int_{\bB}d\mu_E(\xi)\int_{|1-rz\overline{\xi}|<\delta} \frac{d\delta}{\delta^{\lambda+1}} \\
&\lesssim  \int_0^1\mu_E(B(\eta,\delta))\frac{d\delta}{(\delta+1-r+1-\rho)^{\lambda+1}}.
\end{split}\end{equation*}
The fact that we are assuming that $p\geq 2$ gives that $p'-2\leq 0$, and, consequently the above estimates give that
\begin{equation*}\begin{split}&
|{\mathcal V}_{s,p,\lambda}[\mu_E](\rho\eta)|\\&\lesssim\int_0^1  \int_0^1 \mu_E(B(\eta,\delta))^{p'-1}\frac{ (1-r)^{(\lambda+sp-n)(p'-1)}d\delta dr}{(\delta+1-r+1-\rho)^{\lambda+1+\lambda(p'-2)}(1-r)} .
\end{split}\end{equation*}
But
$$
\int_0^1 \frac{(1-r)^{(\lambda-(n-sp))(p'-1)}}{(\delta+1-r)^{\lambda(p'-1)+1}}\frac{dr}{1-r} \lesssim
\frac1{\delta^{(n-sp)(p'-1)+1}}.
$$
Plugging this last estimate in the above one, we obtain
$$
M_{rad}[{\mathcal V}_{s,p,\lambda}[\mu_E]](\eta)\lesssim \int_0^1 \left( \frac{\mu_E(B(\eta,\delta))}{\delta^{n-sp}}\right)^{p'-1} \frac{d\delta}{\delta}={\mathcal W}_{s,p}[\mu_E](\eta),$$
and that ends the proof.
\end{proof}

\section{Applications}\label{sec:applications}

In this section we give some  applications of the above results on multipliers for Hardy-Sobolev spaces.

In the first application we extend a result of Beatrous and Burbea (see \cite{beatrousburbea2}), which proved the same result for $s$ a positive integer.
\begin{prop}
Let  $1<p<\infty$ and $0<s<n/p$.  Then $H^\infty\cap H^{n/s}_s\subset Mult(H^p_s)$.
\end{prop}

\begin{proof}
Let $g\in  H^\infty\cap H^{n/s}_s$.
Since $g$ is bounded, by Theorem \ref{thm:bijectMF}, it is sufficient to prove that $|\PN^{0,s} g|^pd\sigma$ 
is a trace measure for $H^p_s$.

Since $\PN^{0,s}$ is a bijective operator from $H^{n/s}_s$ to $H^{n/s}$ (see Proposition \ref{prop:PNMbij}), it is sufficient to prove that if $h\in H^{n/s}$, then 
$|h|^pd\sigma$ is a trace measure for $H^p_s$.

Let  $r=n/s>p$ and let $q$ be a real number satisfying $\frac{-n}{q}=s-\frac{n}{p}$, that is $1/p=1/q+1/r$. 
Since $s>0$, we have $q>p$ and, by Proposition \ref{prop:embedF}, \eqref{item:embedF4} $H^p_s\subset H^q$. 

Thus, if $h\in H^{n/s}$ and $f\in H^p_s\subset H^q$, H\"older's inequality gives 
$$
\|hf\|_{H^p}\le \|h\|_{H^{n/s}}\, \|f\|_{H^q}\lesssim \|h\|_{H^{n/s}}\, \|f\|_{H^p_s}
$$ 
which proves the result.
\end{proof}

The next application shows that there exists a strong solution of the Corona Theorem for multipliers of $H_s^p$ for some particular data.
\begin{prop}\label{prop:corona}
Let $1<p<\infty$, $0<n-sp<1$ and assume in addition that $n-s<1$ if $p< 2$.  Let $K_i$, $i=1,\dots,l$ be compact subsets of $\bB$ such that $\bigcup_{i=1}^l K_i=\bB$.  Let $V_i$ $i=1,\dots,l$ be the potential multipliers for $H_s^p$ given in Theorem \ref{thm:multsextremal} associated to the extremal measures of the compact sets $K_i$.  We then have that
there there exist multipliers of $H_s^p$, $g_i$, $i=1,\dots,l$,  such that
$$1=\sum_{i=1}^l V_i g_i.$$
That is, there exists a strong solution of the Corona Theorem for multipliers of $H_s^p$ with data $V_1,\dots, V_l$.
\end{prop}
\begin{proof}
Since $V_i\in H_s^p$, for each $i=1,\dots,l$, there exists a.e $V_i^*(\eta)=\lim_{r\rightarrow 1} V_i(r\eta)\in H^p$. On the other hand, by (iii) in Proposition \ref{prop:extremalcapacity}, for almost every $\eta\in K_i$,   ${\rm Re}V_i^*(\eta)\gtrsim 1$. In addition,  Proposition \ref{prop:pgreaterthan2} gives that ${\rm Re}V_i^*(\eta)\geq 0$. Consequently,
$${\rm Re}\sum_{i=1}^lV_i^*(\eta)\gtrsim 1,$$
a.e. $\eta\in\bB$.
Thus, if $P(z,\eta)$ denotes de Poisson-Szeg\"o kernel,
$${\rm Re}\sum_{i=1}^lV_i
(z)=\int_{\bB} P(z,\eta){\rm Re}\sum_{i=1}^lV_i^*(\eta)d\sigma(\eta) \gtrsim 1,$$
for any $z\in \B$.
This estimate gives that
$V:=\frac1{\sum_{i=1}^lV_i}\in H^\infty$. Since $V_i$, $i=1,\dots, l$ are multipliers for $H_s^p$, we deduce from Theorem \ref{thm:charmult}, that
$V$ is also multiplier for $H_s^p$. Thus $1=\sum_{i=1}^l V_i V$.

\end{proof}
Our last application shows that under the same hypothesis than before, for  compact sets $K\subset \bB$ of nonisotropic Riesz capacity zero there exists a sequence $(m_k)_k$ of multipliers for $H_s^p$ which converges in $H_s^p$, and such that $\liminf_{\rho\rightarrow 1} |m_k(\rho(\eta))|=\infty$ for any $\eta\in K$. In this sense, we could say that such compact sets are weak exceptional sets for the multipliers of $H_s^p$.
\begin{prop}\label{prop:compact}
Let $1<p<\infty$, $0<n-sp<1$. Assume in addition that $n-s<1$ if $p< 2$. Let $K\subset \bB$ be a compact set such that $C_{s,p}(K)=0$. We then have that there exists a sequence $(m_k)_k$ of multipliers of $H_s^p$,   such  that 
\begin{enumerate}
\item The sequence $(m_k)_k$ converges to a function $F$ in $H_s^p$.
\item There exists $C>0$ such that $\liminf_{\rho\rightarrow1} |m_k(\rho\eta)|\geq C k$, for any $\eta\in K$.
\end{enumerate}

\end{prop}
\begin{proof}
For any $k$, let $G_k\subset \bB$ be an open set satisfying that $K\subset G_k$ and $C_{s,p}(G_k)\leq \frac1{2^k}$.
Let $\mu_k$ be the extremal potential capacity associated to $G_k$ and $F_k$ the corresponding holomorphic potential given in \eqref{holompot1} and \eqref{holompot2}. By Theorem \ref{thm:multsextremal}, $F_k$ is a multiplier for $H_s^p$, and we also have (see \cite{aherncohn}), that  $$\|F_k\|_{H_s^p}\lesssim C_{s,p}(G_k)\leq \frac1{2^k}.$$

For $k\geq 1$, we define $m_k:= \sum_{i=1}^kF_i$. These functions verify the required properties, since by Proposition \ref{prop:pgreaterthan2}, for each $\eta \in K$ we have
$$\liminf_{\rho\rightarrow1}{\rm Re}F_k(\rho\eta)\gtrsim {\mathcal W}_{s,p}(\mu_k)(\eta)\gtrsim 1,$$
which ends the proof.
\end{proof}

\section{Proof of Theorem \ref{thm:mf}} \label{sec:mf}

\subsection{A Taylor's formula with explicit error term}

\begin{lem} \label{lem:TaylorKernel}
Let $k,m$ be positive integers and let $z,w,u\in \B$. Then we have
\begin{align*}
\frac{1}{(1-z\overline u)^{m}}&
=\sum_{j=0}^k \overline{R^{m-1}_{1,u}}\frac {((z-w)\overline u)^{j}}{(1-w\overline u)^{j+1}} 
+\overline{R^{m-1}_{1,u}}\frac{((z-w)\overline u)^{k+1}}{(1-z\overline u)(1-w\overline u)^{k+1}}
\end{align*}
\end{lem} 

\begin{proof}  For $\lambda$ and $\kappa$  in the open unit disk of $\C$ we have
\begin{align*}
\frac{1}{1-\lambda}
&=\sum_{j=0}^k \frac {(\lambda-\kappa)^{j}}{(1-\kappa)^{j+1}} +\frac{(\lambda-\kappa)^{k+1}}{(1-\lambda)(1-\kappa)^{k+1}}.
\end{align*}

Taking $\lambda=z\overline u$ and $\kappa=w\overline u$, we obtain the case $m=1$.

Since 
$\displaystyle{ \overline{R^{m-1}_{1,u}}\frac{1}{1-z\overline u}=\frac{1}{(1-z\overline u)^{m}}}$, 
the case $m>1$ follows from the case $m=1$.
\end{proof}

\begin{lem}[Taylor's formula] \label{lem:Taylorf}
Let $L,k,l$ be nonnegative integers and $f\in B^1_{-L}$. If $w\in\B$, then we have  
\begin{align*}
f(z)=\sum_{j=0}^k\sum_{|\alpha|=j}\frac{1}{\alpha!}\p^\alpha f(w) (z-w)^\alpha+E^k_{L,l}(f)(z,w)
\end{align*}
where
$$
E^k_{L,l}(f)(z,w):= c_{L+l}\int_{\B}R^l_{n+L} f(u) (1-|u|^2)^{L+l-1} E^k_{L}(z,w,u) d\nu(u)
$$
and
$$
E^k_{L}(z,w,u):=\overline{R^{n+L-1}_{1,u}}\frac{((z-w)\overline u)^{k+1}}{(1-z\overline u)(1-w\overline u)^{k+1}}.
$$
\end{lem}

\begin{proof}
Since 
$$
f(z)=c_{L+l}\int_\B R^l_{n+L}f(u)\frac{(1-|u|^2)^{L+l-1}}{(1-z\overline u)^{n+L}}d\nu(u),
$$ 
 Lemma \ref{lem:TaylorKernel} with $m=n+L$  and the uniqueness of the $k$-th order Taylor's polynomial prove the result. 
\end{proof}

\subsection{Integration by parts formulas} 

\begin{prop}\label{prop:intparts}
Let $N,M$ be real numbers satisfying $n+M>1$ and $N>1$.
If $g\in B^1_{-N}$, then
\begin{align*}
\int_{\B}&g(w)\frac{(1-|w|^2)^{N-1}\,(z_j-w_j)}{(1-z\overline w)^{n+M}}d\nu(w)\\
&=\frac{N-M}{n+M-1}\int_{\B}g(w)\frac{(1-|w|^2)^{N-1}\, w_j}{(1-z\overline w)^{n+M-1}}d\nu(w).
\end{align*}
\end{prop}

\begin{proof}
If $g\in B^1_{-N}$, then for $L>N$ we have 
$$
\sup_{0<r<1} |g(rz)|\le \sup_{0<r<1} |\PP^{L}(g)(rz)|\lesssim |\PP^{L}(g)(z)|\in L^1(d\nu_N).
$$ 
Therefore, by the dominated convergence theorem, it is enough to prove the result for 
functions in $H(\overline \B)$.

An easy computation shows that
\begin{align*}
(\overline\p_j&-w_j\overline R)
\frac{1}{(1-z\overline w)^{n+M-1}}\\
&=(n+M-1)\frac{z_j-w_j}{(1-z\overline w)^{n+M}}+(n+M-1)\frac{w_j}{(1-z\overline w)^{n+M-1}}
\end{align*} 
and
$$
(\overline\p_j-w_j\overline R)(1-|w|^2)^{N-1}=-(N-1)w_j(1-|w|^2)^{N-1}.
$$

Since $N>1$, $(1-|w|^2)^{N-1}$ vanishes on $\bB$ and thus, by integration by parts, we obtain
\begin{align*}
(n+M-1)\int_{\B}&g(w)\frac{(1-|w|^2)^{N-1}\,(z_j-w_j)}{(1-z\overline w)^{n+M}}d\nu(w)\\
=&-(n+M-1)\int_{\B}g(w)\frac{(1-|w|^2)^{N-1}\, w_j}{(1-z\overline w)^{n+M-1}}d\nu(w)\\
&-\int_{\B} \frac{(\overline\p_j-w_j\overline R)\left(g(w)(1-|w|^2)^{N-1}\right)}
{(1-z\overline w)^{n+M-1}}d\nu(w)\\
&+n\int_{\B}g(w)\frac{(1-|w|^2)^{N-1}\, w_j}{(1-z\overline w)^{n+M-1}}d\nu(w)\\
=&(N-M)\int_{\B}g(w)\frac{(1-|w|^2)^{N-1}\, w_j}{(1-z\overline w)^{n+M-1}}d\nu(w),
\end{align*}
which concludes the proof.
\end{proof}

Iterating the above formula we obtain.

\begin{cor} \label{cor:intparts}
Let $N,M$ be real numbers and let  $\alpha$ be a multiindex satisfying  $N>1$ and  $1\le |\alpha|<n+M$. 
If $g\in B^1_{-N}$, then
\begin{align*}
\int_{\B}&g(w)\frac{(1-|w|^2)^{N-1}\,(z-w)^\alpha}{(1-z\overline w)^{n+M}}d\nu(w)\\
&=c_{N,M,|\alpha|}\int_{\B}g(w)\frac{(1-|w|^2)^{N-1}\, w^\alpha}{(1-z\overline w)^{n+M-|\alpha|}}d\nu(w),
\end{align*}
where $c_{N,M,|\alpha|}=\dfrac{(N-M)\cdots(N-M+|\alpha|-1)}{(n+M-1)\cdots(n+M-|\alpha|)}$.
\end{cor}

\subsection{A fractional Leibnitz's type formula}
Combining the above integration by parts formulas and the Taylor's formula  we can prove 
the next two propositions. 

\begin{prop}\label{prop:formula}
Let $\tilde N>1$, $\tilde M>1-n$, $0<t<\tilde N$ and $g\in B^\infty_{-t}$. Then, for any nonnegative integers
$L,k,l$ satisfying $k<\min\{\tilde N-t,n+\tilde M-1\}$ and  $\tilde N-t<L$,  and $f\in B^1_{k+t-\tilde N}$, we have 
\begin{align*}
f(z) \PN^{\tilde N,\tilde M} (g)(z)=&\sum_{j=0}^k \frac{c_{\tilde N,\tilde M,j}}{j!} \PN^{\tilde N,\tilde M-j}\left(g\,d^j f(R,\overset{(j)}{\cdots},R)\right)(z)\\
&+\tilde Q^{\tilde N,\tilde M,k}(f,g)(z),
\end{align*}
where
\begin{align*}
&\tilde Q^{\tilde N,\tilde M,k}(f,g)(z)\\
&\quad:=c_{L+l}\int_\B R^l_{n+L} f(u)(1-|u|^2)^{L+l-1}K^{\tilde N,\tilde M,k}_L(g)(z,u) d\nu(u)
\end{align*}
\begin{align*}
K^{\tilde N,\tilde M,k}_L(g)(z,u)&:=\int_\B g(w) K^{\tilde N,\tilde M,k}_L(z,w,u)d\nu(w),\\
K^{\tilde N,\tilde M,k}_L(z,w,u)&:=\overline{R^{n+L-1}_{1,u}}
\frac{c_{\tilde N}c_{\tilde N,\tilde M,k+1}(w\overline u)^{k+1}(1-|w|^2)^{\tilde N-1}}
{(1-w\overline u)^{k+1}(1-z\overline w)^{n+\tilde M-k-1}(1-z\overline u)},
\end{align*}
and $d^j f$ denotes the $j$-th differential of $f$, that is 
$d^j\,f(R,\overset{(j)}{\cdots},R)(w)=\sum_{|\alpha|=j}\frac{j!}{\alpha!}w^\alpha\p^\alpha f(w) $.
\end{prop}

\begin{rem}\label{rem:remmf3}
Note that if $N-M$ is a negative integer, 
then the result of the above proposition corresponds to a  Leibnitz type formula.
 For instance, if $\tilde N=2$, $\tilde M=3$ and $k=1$, then $\PN^{2,3}=R^1_{n+2}$ and 
the formula is 
$f R^1_{n+2} g=R^1_{n+2} (fg)-\frac{1}{n+2}gR (f)$. In this case $Q^{2,3,1}(f,g)=0$.

Observe that in the above proposition, 
the index $k$ is upper bounded by a constant depending of $\tilde N, \tilde M$ and $t$, 
and in Theorem \ref{thm:mf} is lower bounded by a constant depending of $N, M$ and $t$. 
This fact seems to be contradictory. However,  
we will apply the above proposition to the case $\tilde N=N+J$ and $\tilde M=M+i$ 
with $J, i$ arbitrarily large.
Therefore the upper boundedness of $k$ in the proposition is not relevant in order to prove   
 Theorem \ref{thm:mf}.
\end{rem} 

\begin{proof} Since $L>\tilde N-t$ we have $B^1_{k+t-\tilde N}\subset B^1_{-L}$, and thus, 
by Lemma \ref{lem:Taylorf}, we have
\begin{align*}
f(z) &\PN^{\tilde N,\tilde M} (g)(z)\\
&=c_{\tilde N}\sum_{j=0}^k\sum_{|\alpha|=j}\frac{1}{\alpha!}
\int_\B \p^\alpha f(w) (z-w)^\alpha g(w)
\frac{(1-|w|^2)^{\tilde N-1}}{(1-z\overline w)^{n+\tilde M}}d\nu(w)\\
&+c_{\tilde N}\int_\B E^k_{L,l}(f)(z,w) g(w)
\frac{(1-|w|^2)^{\tilde N-1}}{(1-z\overline w)^{n+\tilde M}}d\nu(w). 
\end{align*}

By Corollary \ref{cor:intparts}, the first term in the right hand part 
of the equality is equal to 
$$
c_{\tilde N}\sum_{j=0}^k\sum_{|\alpha|=j}\frac{1}{\alpha!}
c_{\tilde N,\tilde M,j}\int_\B \p^\alpha f(w) g(w)
 \frac{(1-|w|^2)^{\tilde N-1}\,w^\alpha}{(1-z\overline w)^{n+\tilde M-j}}d\nu(w),
$$
which coincides with
$$
\sum_{j=0}^k \frac{c_{\tilde N,\tilde M,j}}{j!} 
\PN^{\tilde N,\tilde M-j}\left(g\,d^j\,f(R,\overset{(j)}{\cdots},R)\right)(z)
$$

In order to compute the second term, first note that by Corollary \ref{cor:intparts} we have
\begin{align*}
&\int_\B E^k_{L}(z,w,u) g(w) \PN^{\tilde N,\tilde M}(z,w)d\nu(w)\\
&=c_{\tilde N}\int_B \overline{R^{n+L-1}_{1,u}}\frac{g(w) ((z-w)\overline u)^{k+1}}{(1-z\overline u)(1-w\overline u)^{k+1}} ç
\frac{(1-|w|^2)^{\tilde N-1}}{(1-z\overline w)^{n+\tilde M}}d\nu(w)\\
&=c_{\tilde N}c_{\tilde N,\tilde M,k+1}\int_B \overline{R^{n+L-1}_{1,u}}
\frac{g(w)(w\overline u)^{k+1}}{(1-z\overline u)(1-w\overline u)^{k+1}} ç
\frac{(1-|w|^2)^{\tilde N-1}}{(1-z\overline w)^{n+\tilde M-k-1}}d\nu(w)\\
&=K^{\tilde N,\tilde M,k}_L(g)(z,u).
\end{align*}
Therefore, the identity
\begin{equation}\label{eqn:formula}\begin{split}
&\int_\B E^k_{L,l}(f)(z,w) g(w) \PN^{\tilde N,\tilde M}(z,w)d\nu(w)\\
&=c_{L+l}\int_\B R^l_{n+L} f(u)(1-|u|^2)^{L+l-1}K^{\tilde N,\tilde M,k}_L(g)(z,u) d\nu(u),
\end{split}\end{equation}
will be an easy consequence of Fubini's theorem and the fact that $f\in B^1_{k+t-\tilde N}$. 
\end{proof}

The next proposition provides an estimate of $|(1+R)^m \tilde Q^{\tilde N,\tilde M,k}(f,g)(z)|$ for some especial values of $L,\tilde N,\tilde M,l,k$ and $m$.

\begin{prop} \label{prop:estQ}
Let $\tilde N>1$, $0<t<\tilde N$, $\tilde M>1-n$ and $g\in B^\infty_{-t}$.   
Let also $L,k,l$ and $Q^{\tilde N,\tilde M,k}(f,g)$ as in Proposition \ref{prop:formula}. 

If $0< \tilde N-t-M+k+1<m$ and $f\in B^1_{k+t-\tilde N}$, then  we have: 
\begin{align*}
&|(1+R)^m \tilde Q^{\tilde N,\tilde M,k}(f,g)(z)|\\
&\lesssim \|g\|_{B^\infty_{-t}}\left(\int_\B |R^l_{n+L} f(u)|
\frac{(1-|u|^2)^{\tilde N-t-k+l-1}}{|1-z\overline u|^{n+\tilde M-k+m}} d\nu(u)\right.\\
&\left.+(1-|z|^2)^{k+1-m}\Omega_{N-t-\tilde M}(1-|z|^2)
\int_\B |R^l_{n+L} f(u)|\frac{(1-|u|^2)^{L+l-1}}{|1-z\overline u|^{n+L+k+1}}d\nu(u)\right).
\end{align*}
\end{prop}

\begin{proof}
Since 
$$
K^{\tilde N,\tilde M,k}_L(z,w,u)=\overline{R^{n+L-1}_{1,u}}
\frac{c_{\tilde N}c_{\tilde N,\tilde M,k+1}(w\overline u)^{k+1}(1-|w|^2)^{\tilde N-1}}
{(1-w\overline u)^{k+1}(1-z\overline w)^{n+\tilde M-k-1}(1-z\overline u)},
$$
an easy computation shows that 
\begin{align*}
&|(1+R)^m K^{\tilde N,\tilde M,k}_L(f,g)(z)|\\
&=\left|(1+R_z)^m \int_\B g(w)\, K^{\tilde N,\tilde M,k}_L(z,w,u) d\nu(w)\right|\\
&\lesssim \frac{\|g\|_{B^\infty_{-t}}}{|1-z\overline u|}\int_\B\frac{(1-|w|^2)^{\tilde N-t-1}}
{|1-w\overline u|^{n+L+k}|1-z\overline w|^{n+\tilde M-k-1+m}}d\nu(w)\\
&+ \frac{\|g\|_{B^\infty_{-t}}}{|1-z\overline u|^{m+1}}\int_\B\frac{(1-|w|^2)^{\tilde N-t-1}}
{|1-w\overline u|^{n+L+k}|1-z\overline w|^{n+\tilde M-k-1}}d\nu(w)\\
&+ \frac{\|g\|_{B^\infty_{-t}}}{|1-z\overline u|^{n+L}}\int_\B\frac{(1-|w|^2)^{\tilde N-t-1}}
{|1-w\overline u|^{k+1}|1-z\overline w|^{n+\tilde M-k-1+m}}d\nu(w)\\
&+ \frac{\|g\|_{B^\infty_{-t}}}{|1-z\overline u|^{n+L+m}}\int_\B\frac{(1-|w|^2)^{\tilde N-t-1}}
{|1-w\overline u|^{k+1}|1-z\overline w|^{n+\tilde M-k-1}}d\nu(w)\\
&:=I_1+I_2+I_3+I_4.
\end{align*}

Lemma \ref{lem:estP} and the hypotheses  $k<\min\{\tilde N-t,n+\tilde M-1\}$, 
$\tilde N-t<L$,  if $0< \tilde N-t-\tilde M+k+1<m$, we have 
$\max\{\tilde M-k-1,k\}<\tilde N-t<\min\{\tilde M,L+k,\tilde M-k-1+m\}$ gives
\begin{align*}
I_1+I_2+I_3
&\lesssim\|g\|_{B^\infty_{-t}}\left(\frac{(1-|u|^2)^{\tilde N-t-L-k}}
{|1-z\overline u|^{n+\tilde M-k+m}}
+\frac{(1-|z|^2)^{\tilde N-t-\tilde M+k+1-m}}{|1-z\overline u|^{n+L+k+1}}\right)\\
&\qquad+\|g\|_{B^\infty_{-t}}
\frac{(1-|u|^2)^{\tilde N-t-L-k}}{|1-z\overline u|^{n+\tilde M-k+m}}\\
&\qquad+\|g\|_{B^\infty_{-t}}
\frac{(1-|z|^2)^{\tilde N-t-\tilde M+k+1-m}}{|1-z\overline u|^{n+L+k+1}}\\
&\lesssim \|g\|_{B^\infty_{-t}}
\left(\frac{(1-|u|^2)^{\tilde N-t-L-k}}{|1-z\overline u|^{n+\tilde M-k+m}} 
+\frac{(1-|z|^2)^{\tilde N-t-\tilde M+k+1-m}}{|1-z\overline u|^{n+L+k+1}} \right)
\end{align*}
The last inequality is a consequence of $1-|u|^2\le 2 |1-z\overline u|$. 

The last integral depends of the sign of $\tilde N-t-\tilde M$ and in this case Lemma \ref{lem:estP} gives
\begin{align*}
I_4&\lesssim \|g\|_{B^\infty_{-t}}\frac{\Omega_{\tilde N-t-\tilde M}
(|1-z\overline u|)}{|1-z\overline u|^{n+L+m}}
&\lesssim \|g\|_{B^\infty_{-t}}\frac{(1-|z|^2)^{k+1-m}
\Omega_{\tilde N-t-\tilde M}(1-|z|^2))}{|1-z\overline u|^{n+L+k+1}}.
\end{align*}

 Since $(1-|z|^2)^{\tilde N-t-\tilde M}\le \Omega_{\tilde N-t-\tilde M}(1-|z|^2)$, 
 \begin{align*}
 I_1+I_2+I_3+I_4&\lesssim \|g\|_{B^\infty_{-t}}
 \frac{(1-|u|^2)^{\tilde N-t-L-k}}{|1-z\overline u|^{n+\tilde M-k+m}} \\
 &+\|g\|_{B^\infty_{-t}}\frac{(1-|z|^2)^{k+1-m}
 \Omega_{\tilde N-t-\tilde M}(1-|z|^2))}{|1-z\overline u|^{n+L+k+1}},
\end{align*}
 we ends the proof.
\end{proof}

\subsection{Proof of Theorem \ref{thm:mf}}
By Proposition \ref{prop:intpartsphi} for any nonnegative integer $J$ there exist constants $a_i$, $0\le i\le J$ such that 
\begin{align*}
\PN^{N,M}(g)(z)&=c_{N+J}\int_\B g(w)\overline{R^J_{n+N}}\frac{1}{(1-z\overline w)^{n+M}} (1-|w|^2)^{N+J-1}d\nu(w)\\
&=\sum_{i=0}^{J}a_i \PN^{N+J,M+i}(g)(z).
\end{align*}

In order to prove to prove Theorem \ref{thm:mf}, we apply Propositions \ref{prop:formula} and 
\ref{prop:estQ} to the terms $f\PN^{N+J,M+i}(g)$ with  $k+1-n-M<i\le J$. 
 Since  $f\PN^{N+J,M+i}(g)$ satisfy  the conditions in Proposition \ref{prop:formula}, we have 
\begin{align*}
f \PN^{N+J,M+i} (g)=&\sum_{j=0}^k \frac{c_{N+J,M+i,j}}{j!} 
\PN^{N+J,M+i-j}\left(g\,d^j\,f(R,\overset{(j)}{\cdots},R)\right)\\
&+\tilde Q^{N+J,M+i,k}(f,g).
\end{align*}

Defining 
$$
Q^{N,M,k}(f,g)(z):=\sum_{k+1-n-M<i\le J} \tilde Q^{N+J,M+i,k}(f,g)(z)
$$
we obtain formula \eqref{eqn:mf}.

The estimate of $|(1+R)^m Q^{N,M,k}(f,g)|$ given in Theorem \ref{thm:mf} 
follows from the estimates  of $|(1+R)^m \tilde Q^{N+J,M+i,k}(f,g)|$, for $k+1-n-M<i\le J$. 
By proposition \ref{prop:estQ}, all these estimates are bounded for the one corresponding to the case $i=J$, which coincides with  the one stated in Theorem \ref{thm:mf}.

\end{document}